\documentclass[12pt]{amsart}
\usepackage{geometry}
\usepackage[english]{babel}

\usepackage{amssymb}
\usepackage{amsmath}
\usepackage{amsthm}
\usepackage{graphicx}
\usepackage[colorlinks=true, allcolors=blue]{hyperref}

\makeatletter
\renewcommand\th@plain{\slshape}
\makeatother

\newtheorem{theorem}{Theorem}[section]
\newtheorem{lemma}[theorem]{Lemma}

\newtheorem{proposition}[theorem]{Proposition}

\theoremstyle{remark}

\theoremstyle{definition}
\newtheorem{defn}[theorem]{Definition}
\newtheorem{remark}[theorem]{Remark}

\def\cX{\mathcal X}

\def\Aut{{\rm Aut}}

\def\deg{{\rm deg}}
\def\det{{\rm det}}

\def\Fq2{\mathbb F_{q^2}}

\newcommand{\coloneqq}{\mathrel{:=}}

\begin{document}

	\title[Maximality, Weierstrass Semigroups, and Automorphism Group 
of $Y^{q+1} = X^n(X^n + 1)$]{On the Maximality, Weierstrass Semigroups, and Automorphism Group 
of the Curve $Y^{q+1} = X^n(X^n + 1)$}
	\author{João Paulo Guardieiro}
	\address{Centro de Ciências Exatas e Tecnologia, Av. dos Portugueses, 1966, 65.080-805 São Luís-MA, Brazil.}
	\email{joao.guardieiro@ufma.br}
	
	\author{Yuri da Silva}
	\address{Instituto de Matemática, Estatística e Computação Científica da Universidade Estadual de Campinas, R. Sérgio Buarque de Holanda, 651, 13083-970 Campinas-SP, Brazil.}
	\email{y225979@dac.unicamp.br}
	
	\author{Saeed Tafazolian}
	\address{Instituto de Matemática, Estatística e Computação Científica da Universidade Estadual de Campinas, R. Sérgio Buarque de Holanda, 651, 13083-970 Campinas-SP, Brazil.}
	\email{saeed@unicamp.br}

	\begin{abstract}
		We study the algebraic curve over $\mathbb{F}_{q^2}$ defined by
		\[
		y^{q+1} = x^n(x^n+1),
		\]
		where $n$ is a positive integer coprime to the characteristic. We first prove (when $q$ is odd) that the nonsingular model of this curve is $\mathbb{F}_{q^2}$-maximal if and only if $n \mid (q+1)$. Writing $n = \frac{q+1}{m}$, we obtain a family of maximal curves parameterized by the divisors $m$ of $q+1$, which extends the previously studied case $m=3$ corresponding to maximal curves with the third largest possible genus.
		
		For this family, we determine the Weierstrass semigroups at several classes of rational points, including those lying above the branch points of the natural projection. These semigroups are described explicitly in terms of $q$ and $m$, and exhibit different behaviors depending on the arithmetic properties of $m$.
		
		Moreover, we determine the full automorphism group of the curve under a mild condition on the characteristic. Our results extend an earlier work on the case $m=3$ and provide new insight into the structure of this family of maximal curves.
	\end{abstract}
\maketitle
    
	\section{Introduction}
	
	Let $\mathbb{F}_{q^2}$ be the finite field with $q^2$ elements, where 
$q$ is a power of a prime $p$. An algebraic curve $\cX$ defined over 
$\mathbb{F}_{q^2}$ is called \emph{$\mathbb{F}_{q^2}$-maximal} if the 
number of its $\mathbb{F}_{q^2}$-rational points achieves the 
Hasse--Weil upper bound
\[
\# \cX(\mathbb{F}_{q^2}) = q^2 + 1 + 2q\, g(\mathcal{X}),
\]
where $g(\cX)$ denotes the genus of $\cX$. Maximal curves over finite 
fields are a central topic in arithmetic geometry, owing to their rich 
geometric structure and applications in coding theory, cryptography, 
and finite geometry.

The Hermitian curve
\[
\mathcal{H}_{q+1} : \quad y^{q+1} = x^{q+1} + 1
\]
is the classical example of a maximal curve over $\mathbb{F}_{q^2}$.
It has genus $g(\mathcal{H}_{q+1}) = q(q-1)/2$ and achieves Ihara's 
bound; that is, any maximal curve over $\mathbb{F}_{q^2}$ satisfies 
$g(\cX) \le q(q-1)/2$, with equality if and only if 
$\cX \cong \mathcal{H}_{q+1}$. Further families of maximal curves can 
be obtained as quotients or subcovers of the Hermitian curve, though 
not every maximal curve arises in this way.

	A broad framework for the construction of maximal curves is given by 
\emph{Kummer extensions} of the projective line of the form
\[
y^m = f(x),
\]
where $f(x) \in \mathbb{F}_{q^2}[x]$ is a separable polynomial and 
$\gcd(m, \deg f) = 1$. This class includes classical examples such as 
the Fermat and Hermitian curves, and their subfamilies often exhibit 
deep arithmetic properties.

In the present paper, we study the family of curves defined by the 
affine equation
\begin{equation}\label{maincurve}
    \overline{\mathcal{X}}_{n,q} : \quad y^{q+1} = x^n(x^n + 1),
\end{equation}
where $n \ge 2$ is an integer coprime to $p$.

The curve~\eqref{maincurve} enjoys a high degree of symmetry, owing 
to the exponent $q+1$ on $y$, which divides the order of the 
multiplicative group $\mathbb{F}_{q^2}^\ast$. This property makes 
$\overline{\mathcal{X}}_{n,q}$ a natural object of study from both 
algebraic and arithmetic perspectives.

	The main goal of this paper is to give a complete \emph{classification} 
of the maximal curves of type~\eqref{maincurve}. We determine explicit 
conditions on $n$ and $q$ under which the nonsingular model of 
$\overline{\mathcal{X}}_{n,q}$ is $\mathbb{F}_{q^2}$-maximal, and we 
compute its main geometric invariants, including the Weierstrass 
semigroups at certain $\mathbb{F}_{q^2}$-rational points. Furthermore, 
we investigate the structure of the Weierstrass semigroups at the 
remaining rational points and, under a mild assumption on the 
characteristic, we determine the full automorphism group of the curve.
	
	In particular, we compute the genus and the ramification structure of 
the natural projection $x \colon \overline{\mathcal{X}}_{n,q} \to 
\mathbb{P}^1$, and we characterize the cases in which 
$\overline{\mathcal{X}}_{n,q}$ is $\mathbb{F}_{q^2}$-covered by the 
Hermitian curve or by a quotient curve thereof. Consequently, we 
obtain new explicit examples of maximal curves.

    From an applications perspective, the explicit determination of Weierstrass semigroups and the exact number of rational points on algebraic curves has significant implications in geometric coding theory and cryptography. In particular, in the construction of algebraic geometry codes (such as Goppa codes), this precise geometric data is essential for designing codes with optimal parameters, including maximized minimum distance and high dimension (see, e.g., \cite{stichtenoth}). By addressing these problems, our results provide the necessary computational and theoretical tools for these engineering applications, aligning with the core interests of coding theory over finite fields.

	The paper is organized as follows. In 
Section~\ref{sec:maximality}, we determine the precise arithmetic 
conditions on $n$ and $q$ under which $\overline{\mathcal{X}}_{n,q}$ 
is $\mathbb{F}_{q^2}$-maximal. In Section~\ref{sec:results}, we 
introduce the necessary notation and collect preliminary results on 
divisors and automorphisms of $\overline{\mathcal{X}}_{n,q}$ needed 
in the subsequent sections. In Section~\ref{sec:Weierstrass}, we 
determine the Weierstrass semigroups at several classes of rational 
points of $\overline{\mathcal{X}}_{n,q}$. In 
Section~\ref{sec:weierstrass-other}, we study the case $n = \frac{q+1}{4}$ and 
discuss how these results extend to the general case. Finally, in 
Section~\ref{sec:automorphisms}, we determine the full automorphism 
group of $\overline{\mathcal{X}}_{n,q}$ under a mild assumption on 
the characteristic.

	\section{\texorpdfstring{Maximality of the Curve $y^{q+1} = x^{n}(x^{n}+1)$}
{Maximality of the Curve}}\label{sec:maximality}

In this section, we determine precisely when the nonsingular model 
of the curve
\[
\overline{\mathcal{X}}_{n,q} :\quad y^{q+1} = x^{n}(x^{n}+1)
\]
is $\mathbb{F}_{q^2}$-maximal. Our result generalizes the results 
of \cite{TT1,TT2}, where special cases of this family were studied.

\begin{theorem}\label{thm:maximal}
Let $q$ be a prime power and let $n \ge 1$ be an integer with 
$\gcd(2n,q)=1$. Then the nonsingular model of the curve
\[
\overline{\mathcal{X}}_{n,q}:\; y^{q+1}=x^{n}(x^{n}+1)
\]
is $\mathbb{F}_{q^2}$\textnormal{-maximal} if and only if 
$n \mid (q+1)$.
\end{theorem}

\begin{proof}
Assume first that $n \mid (q+1)$. Write $q+1 = nm$.

Let $\mathcal{H}$ be the Hermitian curve $y^{q+1}=x^{q+1}+1$ over 
$\mathbb{F}_{q^{2}}$. Consider the morphism
\[
\varphi:\mathcal{H}\longrightarrow \overline{\mathcal{X}}_{n,q},
\qquad (a,b)\mapsto (a^{m},ab).
\]
A straightforward computation shows that $\varphi(a,b)$ satisfies 
$y^{q+1}=x^{n}(x^{n}+1)$. Hence $\overline{\mathcal{X}}_{n,q}$ is 
a quotient of the Hermitian curve. Since $\mathcal{H}$ is 
$\mathbb{F}_{q^{2}}$-maximal and maximality is preserved under 
nonconstant $\mathbb{F}_{q^2}$-rational morphisms \cite{Serre}, 
the nonsingular model of $\overline{\mathcal{X}}_{n,q}$ is also 
$\mathbb{F}_{q^{2}}$-maximal.

\medskip
Conversely, suppose that $\overline{\mathcal{X}}_{n,q}$ is 
$\mathbb{F}_{q^{2}}$-maximal. Consider the hyperelliptic curve
\[
\mathcal{X}_{n} : \quad y^{2} = x^{n}(x^{n} + 1).
\]
It is covered by $\overline{\mathcal{X}}_{n,q}$ via the map
\[
(x,y) \longmapsto \bigl(x,\, y^{(q+1)/2}\bigr).
\]
A change of variables shows that $\mathcal{X}_{n}$ is birational to
\[
\mathcal{C}_{n} : \quad y_{1}^{2} = x_1^{n} + 1,
\]
regardless of the parity of $n$. If $n$ is even, take 
$y_{1} = y / x^{n/2}$. If $n$ is odd, take 
$y_{2} = y / x^{(n-1)/2}$; then the curve is birational to 
$y_2^2 = x^{n+1} + x$, which is in turn birational to 
$y_1^2 = x_1^n + 1$ via the substitution
\[
y_1 = y_2/x^{(n+1)/2}, \qquad x_1 = 1/x.
\]
Hence $\mathcal{X}_{n} \cong \mathcal{C}_{n}$ over 
$\mathbb{F}_{q^{2}}$.

Since maximality passes to quotients, $\mathcal{C}_{n}$ is also 
$\mathbb{F}_{q^{2}}$-maximal. By \cite[Theorem~1]{Tafazolian2}, 
a hyperelliptic curve of the form $y^{2} = x^{n} + 1$ is 
$\mathbb{F}_{q^{2}}$-maximal if and only if $n \mid (q+1)$. 
Therefore $n \mid (q+1)$, completing the proof.
\end{proof}

\begin{remark}
\rm{Suppose that $n \mid (q+1)$ and write $q+1 = mn$ for some positive 
integer $m$. Then the curve is isomorphic to
\[
\mathcal{X}_{m,q} : \quad y^{q+1} + x^{2(q+1)/m} + x^{(q+1)/m} = 0.
\]
This form is particularly convenient, as the special case $m=3$ 
has already been investigated in the literature 
(see \cite[Example~5.1]{KT1} and \cite{BMV, FM}).}
\end{remark}

\begin{remark}
We will not consider certain cases, since they are already well understood: 
$m=1$ (birationally equivalent to the Hermitian curve), 
$m=2$ (the maximal curve with second highest genus, namely $(q-1)^2/4$), 
and $m=3$ (the maximal curve studied in \cite{BMV}).
\end{remark}

	\section{Notation and Preliminary Results}\label{sec:results}
In this section, we introduce the notation and collect the preliminary 
results on divisors and automorphisms of the curve $\mathcal{X}_{m,q}$ 
needed in the subsequent sections.
 Let $\mathcal{X}_{m,q}$ be the curve over $\mathbb{F}_{q^2}$ defined 
by the equation
\[
y^{q+1} + x^{2(q+1)/m} + x^{(q+1)/m} = 0,
\]
where $m$ is a divisor of $q+1$ with $2 < m < q+1$; write $p=\operatorname{char}(\mathbb{F}_{q^2})$. Note that 
$\mathcal{X}_{m,q}$ is a Kummer extension of $\mathbb{P}^1$ of degree 
$q+1$, and that $\mathcal{X}_{m,q}$ is Galois covered by the Hermitian 
curve $\mathcal{H}:u^{q+1}+v^{q+1}+1=0$. Indeed, the map 
$(u,v) \mapsto (u^m,uv)$ defines a covering from $\mathcal{H}$ to 
$\mathcal{X}_{m,q}$ with generic fiber of cardinality $m$, so that 
$[\mathbb{F}_{q^2}(\mathcal{H}):\mathbb{F}_{q^2}(\mathcal{X}_{m,q})]=m$.

Let $\xi\in\mathbb{F}_{q^2}$ be a primitive $(q+1)$-th root of unity. 
One can verify that $G \leq \Aut(\mathcal{H})$, where
$$G \cong \left\{\begin{pmatrix}
    \xi^{k(q+1)/m} & 0 & 0 \\
    0 & \xi^{q+1 - k(q+1)/m} & 0 \\
    0 & 0 & 1
\end{pmatrix} \right\}_{k = 0, \ldots, m-1}.$$
Since $|G| = m$ and $G$ fixes the covering map from $\mathcal{H}$ to 
$\mathcal{X}_{m,q}$, we conclude that 
$\mathbb{F}_{q^2}(\mathcal{H})/\mathbb{F}_{q^2}(\mathcal{X}_{m,q})$ 
is a Galois extension, and that the nonsingular model of 
$\mathcal{X}_{m,q}$ is precisely $\mathcal{H}/G$.

The pre-image of $(0:0:1)$ under the covering map is the set 
$\{(0:\eta:1) \mid \eta^{q+1}+1=0\}$. The group $G$ partitions these 
points into $\frac{q+1}{m}$ orbits. By \cite[Lemma~VII.6.15]{lorenzini}, 
there are $\frac{q+1}{m}$ places centered at $(0:0:1)$. These places 
will be denoted by $P_0^1, \ldots, P_0^{(q+1)/m}$, and we define
$$\mathcal{O}_0 = \{P_0^1, \ldots, P_0^{(q+1)/m}\} \quad 
\text{and} \quad D_0=\sum_{i=1}^{(q+1)/m}P_0^i.$$

Analogously, the number of places centered at $(1:0:0)$ is
$$\begin{cases}
    \dfrac{q+1}{m} & \text{if } m \text{ is odd,} \\[6pt]
    \dfrac{2(q+1)}{m} & \text{if } m \text{ is even.}
\end{cases}$$
These places are denoted by $P_{\infty}^1, \ldots, P_{\infty}^{(q+1)/m}$ 
or $P_{\infty}^1, \ldots, P_{\infty}^{2(q+1)/m}$ according to whether 
$m$ is odd or even, and we define
$$\mathcal{O}_{\infty} = \begin{cases}
    \{P_{\infty}^1, \ldots, P_{\infty}^{(q+1)/m}\} & \text{if } m 
    \text{ is odd,} \\[4pt]
    \{P_{\infty}^1, \ldots, P_{\infty}^{2(q+1)/m}\} & \text{if } m 
    \text{ is even,}
\end{cases}$$
$$D_{\infty} = \begin{cases}
    \displaystyle\sum_{i=1}^{(q+1)/m}P_{\infty}^i & \text{if } m 
    \text{ is odd,} \\[6pt]
    \displaystyle\sum_{i=1}^{2(q+1)/m}P_{\infty}^i & \text{if } m 
    \text{ is even.}
\end{cases}$$

A point $P=(a:b:1) \in \mathcal{X}_{m,q}(\overline{\mathbb{F}_{q^2}})$ 
with $a \neq 0$ will be denoted by $P_{(a,b)}$. Since this point is 
nonsingular, it is the center of a unique place of 
$\overline{\mathbb{F}_{q^2}}(\mathcal{X}_{m,q})$, which will also be 
denoted by $P_{(a,b)}$. By Kummer theory, the only ramified places of 
the extension 
$\overline{\mathbb{F}_{q^2}}(\mathcal{X}_{m,q})/\overline{\mathbb{F}_{q^2}}(x)$ 
are those of $\mathcal{O}_0$ (with ramification index $m$), those of 
$\mathcal{O}_{\infty}$ (with ramification index $m$ or $\frac{m}{2}$, 
according to whether $m$ is odd or even), and $P_{(a,0)}$ with 
$a^{(q+1)/m}+1 = 0$ (with ramification index $m$). By 
\cite[Proposition~3.7.3]{stichtenoth}, we obtain
$$g(\mathcal{X}_{m,q}) = \begin{cases}
    \dfrac{q^2-q+2m-2}{2m} & \text{if } m \text{ is odd,} \\[6pt]
    \dfrac{q^2-2q+2m-3}{2m} & \text{if } m \text{ is even.}
\end{cases}$$

The extension 
$\mathbb{F}_{q^2}(\mathcal{H})/\mathbb{F}_{q^2}(\mathcal{X}_{m,q})$ 
is also a Kummer extension. Indeed, 
$\mathbb{F}_{q^2}(\mathcal{H}) = \mathbb{F}_{q^2}(\mathcal{X}_{m,q})(u,v)$, 
where $u^m = x$ and $v = y/u \in \mathbb{F}_{q^2}(\mathcal{X}_{m,q})(u)$. 
By \cite[Proposition~3.7.3]{stichtenoth}, this extension is unramified 
when $m$ is odd, and when $m$ is even it is ramified only at 
$P_{\infty}^i$, $i = 1, \ldots, 2(q+1)/m$, with ramification index 
$2$ at each place. In particular, for any place 
$P_{(a,b)} \in \mathbb{P}_{\mathbb{F}_{q^2}(\mathcal{X}_{m,q})}$ and 
any place $Q_{(A,B)} \in \mathbb{P}_{\mathbb{F}_{q^2}(\mathcal{H})}$ 
lying above $P_{(a,b)}$, and any function 
$f \in \mathbb{F}_{q^2}(\mathcal{X}_{m,q})$, we have 
$v_{P_{(a,b)}}(f) = v_{Q_{(A,B)}}(f)$. This allows us to compute the 
Weierstrass semigroups at the places $P_{(a,b)}$ using expansions in 
terms of a local parameter at $Q_{(A,B)}$.

Throughout the paper, we denote
$$\mathcal{O} = \mathcal{O}_0 \cup \mathcal{O}_{\infty} \cup 
\mathcal{O}_m,$$
where $\mathcal{O}_m := \{P_{(a,0)} : a^{(q+1)/m} + 1 = 0\}$, and 
$\mathcal{O}' = \{P_{(a,b)} : 2a^{(q+1)/m} + 1 = 0\}$.

Note that, in this article, we consider automorphisms as certain maps acting on the (finite) points in the curve, not as maps acting on the function field.

\begin{proposition}\label{automorfismos}
The automorphism group $\Aut(\mathcal{X}_{m,q})$ contains a subgroup 
$G$ of order $2(q+1)^2/m$, isomorphic to a semidirect product of 
$\mathbb{Z}/2 \ltimes
(\mathbb{Z}/\frac{q+1}{m} \times \mathbb{Z}/(q+1))$, generated by 
$A \cup \{\theta_2\}$, where
$$A := \{\theta_{\gamma,\delta}(x,y)=(\gamma x, \delta y) \mid 
\gamma^{(q+1)/m}=\delta^{q+1}=1\}, \quad 
\theta_2(x,y) = \left(\frac{y^m}{x}, y \right).$$
\end{proposition}

\begin{proof}
The maps $\theta_{\gamma,\delta}$ clearly preserve the equation of the 
curve, as does $\theta_2$, since
\[
y^{q+1} + \left(\frac{y^m}{x}\right)^{2(q+1)/m} + 
\left(\frac{y^m}{x}\right)^{(q+1)/m}
=
y^{q+1} + \frac{y^{2(q+1)}}{x^{2(q+1)/m}} + 
\frac{y^{q+1}}{x^{(q+1)/m}}
\]
\[
=
\frac{y^{q+1}}{x^{2(q+1)/m}} \left( x^{2(q+1)/m} + y^{q+1} + 
x^{(q+1)/m} \right) = 0.
\]
Moreover, $\theta_2^2 = \mathrm{id}$, and
\[
\theta_2(\theta_{\gamma,\delta}(x,y)) =
\theta_2(\gamma x, \delta y) =
\left(\frac{\delta^m}{\gamma} \cdot \frac{y^m}{x}, \delta y\right) =
\theta_{\delta^m/\gamma,\,\delta}(\theta_2(x,y)),
\]
with $(\delta^m/\gamma)^{(q+1)/m} = \delta^{q+1}/\gamma^{(q+1)/m} = 1$. 
Hence $G$ is the semidirect product of $\{\mathrm{id},\theta_2\}$ 
acting on $A$.
\end{proof}

The elements of $A$ act transitively on $\mathcal{O}_0$, 
$\mathcal{O}_{\infty}$, and $\mathcal{O}_m$, while $\theta_2$ maps 
$\mathcal{O}_0$ to $\mathcal{O}_m$. Therefore,
$$H(P_0^i) = H(P_{(a,0)}), \quad i = 1, \ldots, \frac{q+1}{m},\quad 
a^{(q+1)/m}+1=0.$$

\begin{proposition}\label{divi-xy1}
Let $m$ be odd and $a, b \in \mathbb{F}_{q^2}^*$. We have
$$(x-a) = \begin{cases}
    \displaystyle\sum_{\xi^{q+1}=1} P_{(a,\xi b)} - m D_{\infty} 
    & \text{if } a^{(q+1)/m}+1\neq 0, \\[6pt]
    (q+1)P_{(a,0)} - m D_{\infty} 
    & \text{if } a^{(q+1)/m}+1 = 0,
\end{cases}$$
$$(x) = m\sum_{i=1}^{(q+1)/m}P_0^i - m D_{\infty},$$
$$(y) = \sum_{i=1}^{(q+1)/m}P_0^i + 
\sum_{a^{(q+1)/m}+1=0}P_{(a,0)} - 2D_{\infty},$$
$$(y-b) = E_b - 2D_{\infty},$$
where $E_b$ is an effective divisor of degree $\frac{2(q+1)}{m}$, with
$$\mathrm{Supp}(E_b) = 
\{P_{(a,b)} \mid b^{q+1}+a^{2(q+1)/m}+a^{(q+1)/m}=0\}$$
and
$$v_{P_{(a,b)}}(E_b) = \begin{cases}
    2 & \text{if } 2a^{(q+1)/m}+1=0, \\
    1 & \text{otherwise.}
\end{cases}$$
\end{proposition}

\begin{proof}
The formulas for $(x)$ and $(x-a)$ follow from the ramification 
indices of the extension 
$\mathbb{F}_{q^2}(\mathcal{X}_{m,q})/\mathbb{F}_{q^2}(x)$. The formula 
for $(y)$ follows from 
$y^{q+1}=-x^{(q+1)/m}(x^{(q+1)/m}+1)$. For $(y-b)$, we note that 
$a$ is a multiple root of $x^{2(q+1)/m}+x^{(q+1)/m}+b^{q+1}$ if and 
only if $2a^{(q+1)/m}+1=0$.
\end{proof}

\begin{proposition}\label{divi-xy2}
Let $m$ be even and $a, b \in \mathbb{F}_{q^2}^*$. We have
$$(x-a) = \begin{cases}
    \displaystyle\sum_{\xi^{q+1}=1} P_{(a,\xi b)} - 
    \frac{m}{2} D_{\infty} 
    & \text{if } a^{(q+1)/m}+1\neq 0, \\[6pt]
    (q+1)P_{(a,0)} - \frac{m}{2} D_{\infty} 
    & \text{if } a^{(q+1)/m}+1 = 0,
\end{cases}$$
$$(x) = m\sum_{i=1}^{(q+1)/m}P_0^i - \frac{m}{2} D_{\infty},$$
$$(y) = \sum_{i=1}^{(q+1)/m}P_0^i + 
\sum_{a^{(q+1)/m}+1=0}P_{(a,0)} - D_{\infty},$$
$$(y-b) = E_b - D_{\infty},$$
where $E_b$ is an effective divisor of degree $\frac{2(q+1)}{m}$ with
$$\mathrm{Supp}(E_b) = 
\{P_{(a,b)} \mid b^{q+1}+a^{2(q+1)/m}+a^{(q+1)/m}=0\}$$
and
$$v_{P_{(a,b)}}(E_b) = \begin{cases}
    2 & \text{if } 2a^{(q+1)/m}+1=0, \\
    1 & \text{otherwise.}
\end{cases}$$
\end{proposition}

\begin{proof}
The proof follows analogously to that of 
Proposition~\ref{divi-xy1}.
\end{proof}
	
	\section{\texorpdfstring
{Weierstrass semigroups at the places of $\mathcal{O}$}
{Weierstrass semigroups}}\label{sec:Weierstrass}
	In this section, we present the Weierstrass semigroups at the places of $\mathcal{O}$. In the first subsection, we study the genera of some numerical semigroups, and in the second one we determine that these are indeed the Weierstrass semigroups at the places of $\mathcal{O}$, by constructing explicit functions with the desired pole divisor.
	
	\subsection{The genera of some numerical semigroups}

Denote
\[
S(q,m) \coloneqq \left\langle q+1,\, q,\, \ldots,\, q+1-\left\lfloor\frac{m}{2}\right\rfloor \right\rangle
+ \left\langle q+1-m \right\rangle,
\]
\[
T(q,m) \coloneqq \begin{cases}
    \left\langle q+1 \right\rangle
    + \left\langle q, q-2, \ldots, q+1-m \right\rangle
    & \text{if } m \text{ is odd,} \\[4pt]
    \left\langle q, q-1, \ldots, q+1-\tfrac{m}{2} \right\rangle
    + \left\langle \tfrac{q+1}{2} \right\rangle
    & \text{if } m \text{ is even.}
\end{cases}
\]

\begin{proposition}\label{genus-P0}
    The genus of $S(q,m)$ is at most $g = g(\mathcal{X}_{m,q})$.
\end{proposition}

\begin{proof}
    Let $a \coloneqq (q+1)-m = \min\left(S(q,m) \smallsetminus \{0\}\right)$.
    Following the ideas in \cite{AlgorithmFrobeniusTree}, we will present a
    function $f \colon \{0,\ldots,a-1\} \to S(q,m)$ such that, for each
    $v \in \{0,\ldots,a-1\}$, $f(v) \equiv v \pmod{a}$. In particular,
    $S(q,m) \supseteq f(v) + a\mathbb{N}$. We will then obtain
    \begin{align*}
        g\left(S(q,m)\right)
        &= \left|\bigcup_{v=1}^{a-1}
           \left\{ w \in \mathbb{N} \mid w \notin S(q,m),\
           w \equiv v \pmod{a} \right\}\right| \\
        &= \sum_{v=1}^{a-1}
           \left|\left\{ k \in \mathbb{N} \mid v+ka \notin S(q,m)
           \right\}\right| \\
        &\le \sum_{v=1}^{a-1}
             \left|\left\{ k \in \mathbb{N} \mid v+ka < f(v)
             \right\}\right| \\
        &= \sum_{v=1}^{a-1} \left\lfloor \frac{f(v)}{a} \right\rfloor,
    \end{align*}
    and $f$ will be chosen so that
    $\sum_{v=1}^{a-1}\lfloor\frac{f(v)}{a}\rfloor = g$.

    Define $f(0) = 0$. For $v \in \{1,\ldots,a-1\}$, we define $f(v)$
    by cases.

    \begin{itemize}
        \item \textbf{Suppose that $v \ge m - \lfloor\frac{m}{2}\rfloor$.}
        Denote $k \coloneqq \lceil\frac{v}{m}\rceil \ge 1$.
        We claim that $v + ka \in S(q,m)$.
        It follows that
        $k\!\left(m-\lfloor\tfrac{m}{2}\rfloor\right) \le v \le km$.
        The second inequality comes from the definition of $k$.
        For the first inequality, if $k=1$, we already have
        $m - \lfloor\frac{m}{2}\rfloor \le v$ by hypothesis.
        Now suppose $k \ge 2$; then
        \begin{align*}
            v - k\!\left(m-\left\lfloor\tfrac{m}{2}\right\rfloor\right)
            \ge v - k\tfrac{m+1}{2}
            = (v-km) + k\tfrac{m-1}{2}
            \ge -(m-1) + k\tfrac{m-1}{2} \ge 0.
        \end{align*}
        Hence there exist $e_0, \ldots, e_{\lfloor m/2 \rfloor} \in \mathbb{N}$
        such that $\sum_{i=0}^{\lfloor m/2 \rfloor} e_i = k$ and
        $\sum_{i=0}^{\lfloor m/2 \rfloor} e_i
        \!\left(m - \lfloor\tfrac{m}{2}\rfloor + i\right) = v$.
        Thus,
        \begin{align*}
            v + ka
            &= \sum_{i=0}^{\lfloor m/2 \rfloor} e_i
               \!\left(m - \left\lfloor\tfrac{m}{2}\right\rfloor + i + a\right)
             = \sum_{i=0}^{\lfloor m/2 \rfloor} e_i
               \!\left((q+1) - \left\lfloor\tfrac{m}{2}\right\rfloor + i\right)
             \in S(q,m).
        \end{align*}
        So we may define $f(v) \coloneqq v + ka$.

        \item \textbf{Suppose that $v < m - \lfloor\frac{m}{2}\rfloor$.}
        We claim that
        $v + \left(\frac{q+1}{m}+1\right)a \in S(q,m)$.
        It follows that
        $\frac{q+1}{m}\!\left(m-\lfloor\tfrac{m}{2}\rfloor\right)
        \le v+a \le \frac{q+1}{m} \cdot m$.
        The second inequality is $v + a < m + a = q+1$.
        For the first inequality, we proceed by cases.
        \begin{itemize}
            \item If $m$ is even, then
            \[
                v + a - \tfrac{q+1}{m}\!\left(m-\left\lfloor\tfrac{m}{2}
                \right\rfloor\right)
                = v + (q+1) - m - \tfrac{q+1}{2}
                \ge \tfrac{q+1}{2} - m \ge 0,
            \]
            since $m$ is a proper divisor of $q+1$.
            \item If $m$ is odd and $m \neq \frac{q+1}{2}$, then,
            using $m > 1$,
            \begin{align*}
                v + a - \tfrac{q+1}{m}\!\left(m-\left\lfloor\tfrac{m}{2}
                \right\rfloor\right)
                &= v + (q+1) - m - (q+1)\tfrac{1+m^{-1}}{2} \\
                &\ge v + (q+1) - m - (q+1)\tfrac{1+3^{-1}}{2} \\
                &= v + \tfrac{q+1}{3} - m \ge 0.
            \end{align*}
            \item If $m$ is odd and $m = \frac{q+1}{2}$, then
            \[
                v + a - \tfrac{q+1}{m}\!\left(m-\left\lfloor\tfrac{m}{2}
                \right\rfloor\right)
                = v + m - 2\cdot\tfrac{m+1}{2} = v - 1 \ge 0.
            \]
        \end{itemize}
        Hence there exist $e_0,\ldots,e_{\lfloor m/2\rfloor} \in \mathbb{N}$
        such that $\sum_{i=0}^{\lfloor m/2\rfloor} e_i = \frac{q+1}{m}$ and
        $$\sum_{i=0}^{\lfloor m/2\rfloor} e_i
        \!\left(m-\lfloor\tfrac{m}{2}\rfloor+i\right) = v+a.$$
        Thus, as in the previous case, we may define
        $f(v) \coloneqq v + \left(\frac{q+1}{m}+1\right)a \in S(q,m)$.
    \end{itemize}

    \vspace{0.3cm}
    Now we have
    \[
        \sum_{v=1}^{a-1} \left\lfloor\frac{f(v)}{a}\right\rfloor
        = \sum_{v=1}^{m-\lfloor m/2\rfloor-1}
          \left(\frac{q+1}{m}+1\right)
        + \sum_{v=m-\lfloor m/2\rfloor}^{a-1}
          \left\lceil\frac{v}{m}\right\rceil.
    \]
    Note that
    \begin{align*}
        \sum_{v=m-\lfloor m/2\rfloor}^{a-1}
        \left\lceil\frac{v}{m}\right\rceil
        &= \sum_{k=1}^{\frac{q+1}{m}-1} k \cdot
           \sum_{v=m-\lfloor m/2\rfloor}^{a-1}
           \left[k = \left\lceil\tfrac{v}{m}\right\rceil\right] \\
        &= 1 \cdot \sum_{v=m-\lfloor m/2\rfloor}^{m} 1
           + \sum_{k=2}^{\frac{q+1}{m}-1} k \cdot
             \sum_{v=(k-1)m+1}^{km - [k=\frac{q+1}{m}-1]} 1 \\
        &= \left(\left\lfloor\tfrac{m}{2}\right\rfloor+1\right)
           + \sum_{k=2}^{\frac{q+1}{m}-1}
             k\!\left(m - \left[k=\tfrac{q+1}{m}-1\right]\right) \\
        &= \left(\left\lfloor\tfrac{m}{2}\right\rfloor+1\right)
           + m\!\left(\binom{(q+1)/m}{2}-1\right)
           - \left(\tfrac{q+1}{m}-1\right).
    \end{align*}
    Therefore,
    {\small
    \begin{align*}
        \sum_{v=1}^{a-1}\left\lfloor\frac{f(v)}{a}\right\rfloor
        &= \left(\frac{q+1}{m}+1\right)
           \left(m-\left\lfloor\tfrac{m}{2}\right\rfloor-1\right)
         + \left(\left\lfloor\tfrac{m}{2}\right\rfloor+1\right)
         + m\!\left(\binom{(q+1)/m}{2}-1\right)
         - \left(\tfrac{q+1}{m}-1\right),
    \end{align*}
    }
    which equals $g$ after straightforward simplification.
\end{proof}

\begin{proposition}\label{genus-Pinf}
    The genus of $T(q,m)$ is at most $g = g(\mathcal{X}_{m,q})$.
\end{proposition}

\begin{proof}
    The proof is similar to the previous one, though the details are more
    involved. Let $a$ be the smallest element of $T(q,m) \smallsetminus \{0\}$.
    When $m$ is even, for $v \ge 1$ we define
    $$f(v) \coloneqq v + \left(2\left\lceil\frac{a-v}{m/2}\right\rceil - 1\right)a.$$
    When $m$ is odd, for $v \ge 1$ we define
    \[
    f(v) \coloneqq \begin{cases}
        v + \left\lceil\dfrac{v}{m}\right\rceil a
            & \left(v \ge m\right), \\[6pt]
        v + a
            & \left(v < m,\ v \in 2\mathbb{N}\right), \\[6pt]
        v + \left(\dfrac{q+1}{m}+1\right)a
            & \left(v < m,\ v \notin 2\mathbb{N}\right).
    \end{cases}
    \]

    Let us show that $f(v) \in T(q,m)$ for each $1 \le v \le a-1$.

    \medskip
    \noindent\textbf{Case: $m$ is even.}
    We have $a = \frac{q+1}{2}$. Write $a - v = k\frac{m}{2} - s$, where
    $k = \lceil\frac{a-v}{m/2}\rceil \ge 1$ and $0 \le s \le \frac{m}{2}-1$.
    We have that $k \le a-v \le k\frac{m}{2}$. Indeed,
    $$a - v - k = k\left(\tfrac{m}{2}-1\right) - s \ge (k-1)\left(\tfrac{m}{2}-1\right) \ge 0,$$
    and $k\frac{m}{2} \ge \frac{a-v}{m/2}\cdot\frac{m}{2}$ by definition of $k$.
    As before, there exist $e_0, \ldots, e_{m/2-1} \in \mathbb{N}$ such that
    $\sum_{i=0}^{m/2-1} e_i = k$ and
    $\sum_{i=0}^{m/2-1} e_i\!\left(\frac{m}{2}-i\right) = a-v$.
    Therefore $T(q,m)$ contains
    $$\sum_{i=0}^{m/2-1} e_i\!\left(q+1-\tfrac{m}{2}+i\right)
      = k(q+1) - a + v = (2k-1)a + v = f(v).$$

    \medskip
    \noindent\textbf{Case: $m$ is odd.}
    We have $a = (q+1)-m$. We show by induction on $v$ that
    $v + \lceil\frac{v}{m}\rceil a \in T(q,m)$ for each
    $v \in \{m+1, \ldots, a-1\}$.
    Write $v = km - s$, with $k = \lceil\frac{v}{m}\rceil$ and $0 \le s \le m-1$.
    We have two cases:

    \begin{itemize}
        \item \textbf{Case $k = 2$.} So $m+1 \le v \le 2m$.
        When $v = 2m$,
        $$v + 2a = 2(a+m) = 2(q+1) \in T(q,m).$$
        Otherwise $v \le 2m-1$.
        If $v$ is odd, write $v = m + 2i$ for some
        $i \in \{1,\ldots,\frac{m-1}{2}\}$; then
        $$v + 2a = (a+m) + (a+2i) = (q+1) + (q+1-(m-2i)) \in T(q,m).$$
        If $v$ is even, write $v = m-1+2i$ for some
        $i \in \{1,\ldots,\frac{m-1}{2}\}$; then
        $$v + a = (a+m-1) + (a+2i) = (q+1-1) + (q+1-(m-2i)) \in T(q,m).$$

        \item \textbf{Case $k \ge 3$.} So $v \ge (k-1)m+1 \ge 2m+1$, hence
        $v - m \ge m+1$. By the inductive hypothesis,
        $(v-m) + (k-1)a \in T(q,m)$, and therefore
        $$v + ka = (v-m) + (k-1)a + (a+m) \in T(q,m).$$
    \end{itemize}

    Thus $f(v) \in T(q,m)$ for each $m+1 \le v \le a-1$.
    When $v = m$, we have $$f(v) = m + a = q+1 \in T(q,m).$$

    If $1 \le v < m$ and $v$ is even, then
    $f(v) = v + a = q+1-(m-v) \in T(q,m)$.

    If $1 \le v < m$ and $v$ is odd, we consider two cases:

    \begin{itemize}
        \item \textbf{Case $\frac{q+1}{2} \neq m$.} Then $v - m + a \le a-1$
        and
        $$v - m + a \ge 1 - m + a = 1 + (q+1) - 2m \ge 1 + 3m - 2m = 1 + m.$$
        So, as shown above, $f(v) = v - m + a + ka \in T(q,m)$, where
        $$k = \left\lceil\frac{v-m+a}{m}\right\rceil
            = \left\lceil\frac{v}{m}\right\rceil - 1 + \frac{a}{m}
            = \frac{a}{m}.$$
        Thus,
        \begin{equation*}
            v + \left(\frac{q+1}{m}+1\right)a
            = v + \left(\frac{a+m}{m}+1\right)a
            = (v-m+a+ka) + (a+m) \in T(q,m).
        \end{equation*}

        \item \textbf{Case $\frac{q+1}{2} = m$.} Then,
        \begin{equation*}
            v + \left(\frac{q+1}{m}+1\right)a = v + 3m
            = (2m-1) + (v+m+1)
            = (q+1-1) + (q+1-(m-1-v)) \in T(q,m).
        \end{equation*}
    \end{itemize}

    It remains to show that $\sum_{v=1}^{a-1}\lfloor\frac{f(v)}{a}\rfloor = g$.

    When $m$ is even,
    \begin{align*}
        \sum_{v=1}^{a-1}\left\lfloor\frac{f(v)}{a}\right\rfloor
        &= \sum_{v=1}^{a-1}\left(2\left\lceil\frac{a-v}{m/2}\right\rceil-1\right) \\
        &= \sum_{k=1}^{\frac{q+1}{m}}\left(2k-1\right)
           \cdot\sum_{v=1}^{a-1}
           \left[k=\left\lceil\tfrac{a-v}{m/2}\right\rceil\right] \\
        &= \sum_{k=1}^{\frac{q+1}{m}}\left(2k-1\right)
           \cdot\sum_{w=1}^{a-1}
           \left[k=\left\lceil\tfrac{w}{m/2}\right\rceil\right] \\
        &= \sum_{k=1}^{\frac{q+1}{m}}\left(2k-1\right)
           \cdot\sum_{w=(k-1)(m/2)+1}^{k(m/2)-[k=\frac{q+1}{m}]}1 \\
        &= \sum_{k=1}^{\frac{q+1}{m}}\left(2k-1\right)
           \cdot\left(\frac{m}{2}-\left[k=\tfrac{q+1}{m}\right]\right) \\
        &= \left(2\binom{(q+1)/m+1}{2}-\frac{q+1}{m}\right)\frac{m}{2}
           - \left(2\,\frac{q+1}{m}-1\right),
    \end{align*}
    which equals $g$ after straightforward simplification.

    When $m$ is odd,
    \begin{align*}
        \sum_{v=1}^{a-1}\left\lfloor\frac{f(v)}{a}\right\rfloor
        &= \sum_{v=m}^{a-1}\left\lceil\frac{v}{m}\right\rceil
           + \sum_{v=1}^{m-1}
             \left(\frac{q+1}{m}+1\right)^{[v\notin 2\mathbb{N}]} \\
        &= 1 + \sum_{k=2}^{\frac{q+1}{m}-1} k
             \cdot\sum_{v=m+1}^{a-1}
             \left[k=\left\lceil\tfrac{v}{m}\right\rceil\right]
           + \left(\frac{q+1}{m}+1\right)\cdot\frac{m-1}{2}
           + 1\cdot\frac{m-1}{2} \\
        &= 1 + m\!\left(\binom{(q+1)/m}{2}-1\right)
           - \left(\frac{q+1}{m}-1\right)
           + \left(\frac{q+1}{m}+1\right)\cdot\frac{m-1}{2}
           + \frac{m-1}{2},
    \end{align*}
    which equals $g$ after straightforward simplification.
\end{proof}

	\subsection{\texorpdfstring
{The Weierstrass semigroups at the places of $\mathcal{O}$}
{The Weierstrass semigroups at the places of O}}

We can now determine the Weierstrass semigroups of the places of $\mathcal{O}$.
This will be done by providing explicit functions with the desired pole divisor.

\begin{proposition}\label{prop:num-zeros}
    Let $P_0^i$, $i=1,\ldots,\frac{q+1}{m}$, be a place centered at $(0:0:1)$,
    and let $a\in\mathbb{F}_{q^2}$ be such that $a^{(q+1)/m}+1=0$. We have
    $$H(P_0^i) = H(P_{(a,0)}) =
      \left\langle q+1-m,\, q+1-\left\lfloor\tfrac{m}{2}\right\rfloor,\,
      \ldots,\, q+1 \right\rangle.$$
\end{proposition}

\begin{proof}
    We already know that $H(P_0^i) = H(P_{(a,0)})$, since these places lie in
    the same orbit under the group $G$ defined in
    Proposition~\ref{automorfismos}. We will then show that the result holds
    for a place $P_{(a,0)}$.

    With the information on the divisors of $x$ and $y$ given in
    Propositions~\ref{divi-xy1} and~\ref{divi-xy2}, one can verify that
    $$\left(\frac{y^i}{x-a}\right)_{\infty} = (q+1-i)P_{(a,0)},
      \quad i=0, \ldots, \left\lfloor \frac{m}{2} \right\rfloor.$$
    Additionally, we have
    $$\left(\frac{y^m}{x(x-a)}\right) =
      m\sum_{\substack{\overline{a}^{(q+1)/m}+1 = 0 \\ \overline{a} \neq a}}
      P_{(\overline{a},0)} - (q+1-m)P_{(a,0)}.$$
    Therefore,
    $$\left\langle q+1-m,\, q+1-\left\lfloor\tfrac{m}{2}\right\rfloor,\,
      \ldots,\, q+1 \right\rangle \subseteq H(P_{(a,0)}).$$
    From Proposition~\ref{genus-P0}, we conclude that $S(q,m)$ and
    $H(P_{(a,0)})$ have the same genus, and since one contains the other,
    the result follows.
\end{proof}
The Weierstrass semigroups at the places of $\mathcal{O}_{\infty}$ depend on
the parity of $m$.

\begin{lemma}\label{Pinf-m-even}
    Let $m$ be even and let $P_{\infty}^i$ be a place centered at $(1:0:0)$.
    There exist a local parameter $\pi$ at $P_{\infty}^i$ and elements
    $\alpha_i, \beta_i \in \mathbb{F}_{q^2}^*$ such that
    $\alpha_i^{2(q+1)/m} = -1$ and
    $$x = \alpha_i \pi^{-m/2} + \beta_i \pi^{(q+1-m)/2} + \mathrm{h.o.t.}$$
\end{lemma}

\begin{proof}
    From Proposition~\ref{divi-xy2}, we can choose $\pi = \frac{1}{y}$ for any
    $i = 1, \ldots, \frac{2(q+1)}{m}$. Setting $y = \pi^{-1}$, we expand
    $x = \alpha_i \pi^{-m/2} + a_{-m/2+1}\pi^{-m/2+1} + \cdots$.
    Note that $m$ even implies $p \neq 2$.
    From the equation of the curve, we obtain
    \begin{align*}
        \pi^{-(q+1)}
        &+ \left(\alpha_i \pi^{-m/2} + a_{-m/2+1}\pi^{-m/2+1} + \cdots
          \right)^{\frac{2(q+1)}{m}} \\
        &+ \left(\alpha_i \pi^{-m/2} + a_{-m/2+1}\pi^{-m/2+1} + \cdots
          \right)^{\frac{q+1}{m}} = 0.
    \end{align*}
    By expanding the binomial terms, we get
    \begin{align*}
        \pi^{-(q+1)}
        &+ \Bigl(\alpha_i^{2(q+1)/m}\pi^{-(q+1)}
           + \tfrac{2(q+1)}{m}\,\alpha_i^{2(q+1)/m-1}a_{-m/2+1}
             \pi^{-(q+1)+1} + \textrm{h.o.t.}\Bigr) \\
        &+ \bigl(\alpha_i^{(q+1)/m}\pi^{-(q+1)/2} + \textrm{h.o.t.}\bigr)
         = 0.
    \end{align*}
    We thus conclude that $\alpha_i^{2(q+1)/m} = -1$. Also, the smallest
    exponent of $\pi$ in the second binomial expansion is $-\frac{q+1}{2}$,
    so we must have
    $$\frac{2(q+1)}{m}\,\alpha_i^{2(q+1)/m-1}a_{-m/2+1} = 0
      \Longrightarrow a_{-m/2+1} = 0.$$
    Consequently, the term with the second smallest exponent of $\pi$ in the
    first binomial expansion is
    $\frac{2(q+1)}{m}\,\alpha_i^{2(q+1)/m-1}a_{-m/2+2}\pi^{-(q+1)+2}$,
    and we also obtain $a_{-m/2+2} = 0$. Proceeding inductively, we get
    $$a_{-m/2+1} = \cdots = a_{(q+1-m)/2-1} = 0.$$
    Finally, we define $\beta_i \coloneqq a_{(q+1-m)/2}$, which satisfies
    $$\frac{2(q+1)}{m}\,\alpha_i^{2(q+1)/m-1}\beta_i
      + \alpha_i^{(q+1)/m} = 0. $$
\end{proof}

We will now obtain the Weierstrass semigroups at the places centered at
$(1:0:0)$. We first note that $\alpha_i \neq \alpha_j$ whenever $i \neq j$.
Indeed, as observed in the determination of $\beta_i$ in
Lemma~\ref{Pinf-m-even}, the expansion of $x$ in terms of $\pi$ is uniquely
determined by $\alpha_i$. Since any element of
$\mathbb{F}_{q^2}(\mathcal{X}_{m,q})$ can be written as a rational function
in $x$ and $y$, having $\alpha_i = \alpha_j$ would imply that the rings of
regular functions at $P_{\infty}^i$ and $P_{\infty}^j$ coincide, hence
$P_{\infty}^i = P_{\infty}^j$.

\begin{proposition}\label{prop:num-inf-even}
    Let $m$ be even and let $P_{\infty}^i$, $i = 1, \ldots, \frac{2(q+1)}{m}$,
    be a place centered at $(1:0:0)$. We have
    $$H(P_{\infty}^i) =
      \left\langle \frac{q+1}{2},\, q+1 - \frac{m}{2},\, \ldots,\, q+1
      \right\rangle.$$
\end{proposition}

\begin{proof}
    We already noted in Proposition~\ref{automorfismos} that all these places
    have the same Weierstrass semigroup. Therefore, we show the result for
    $i = 1$. Let
    $$F(T) = b_0 + b_1 T + \cdots + b_{2(q+1)/m-1} T^{2(q+1)/m-1}$$
    be the polynomial of degree $\frac{2(q+1)}{m}-1$ whose roots are
    $\alpha_2, \ldots, \alpha_{2(q+1)/m}$. Define
    $$f = \sum_{i=0}^{2(q+1)/m-1} b_i\, x^i\, y^{q+1-(i+1)m/2}.$$
    By expanding in terms of $\pi$, we conclude that
    $$f = F(\alpha_i)\pi^{-(q+1-m/2)} + G(\alpha_i,\beta_i)\pi^{-(q+1-m)/2}
        + \textrm{h.o.t.}$$
    Therefore, $v_{P_{\infty}^1}(f) = -(q+1-\frac{m}{2})$, and
    $v_{P_{\infty}^i}(f) \geq -\frac{q+1-m}{2}$ for
    $i = 2, \ldots, \frac{2(q+1)}{m}$.

    By the strict triangle inequality for valuations,
    $v_{P_0^i}(f) = q+1-\frac{m}{2}$ and $f$ has no poles outside
    $\mathcal{O}_{\infty}$. Comparing degrees, we conclude that
    $$(f) = \left(q+1-\tfrac{m}{2}\right)\sum_{i=1}^{(q+1)/m} P_0^i
           - \left(q+1-\tfrac{m}{2}\right) P_{\infty}^1
           - \tfrac{q+1-m}{2} \sum_{i=2}^{2(q+1)/m} P_{\infty}^i.$$
    Setting $g_1 \coloneqq \frac{f}{x^{(q+1)/m-1}}$, we have
    $$(g_1) = \frac{m}{2}\sum_{i=1}^{(q+1)/m} P_0^i
            - \frac{q+1}{2} P_{\infty}^1.$$
    Analogously, we define a function $g_i$ whose divisor is
    $(g_i) = \frac{m}{2}\sum_{j=1}^{(q+1)/m} P_0^j
            - \frac{q+1}{2} P_{\infty}^i$.
    We then obtain
    $$\left(\frac{y^j \cdot f}{g_2 \cdots g_{2(q+1)/m}}\right)_{\infty}
      = \left(q+1-\tfrac{m}{2}+j\right) P_{\infty}^1.$$
    Finally, since $P_{\infty}^1$ is a rational place, we have
    $q+1 \in H(P_{\infty}^1)$, and the Riemann--Roch theorem gives
    $(q+1)P \sim (q+1)P_{\infty}^1$ for any rational place $P$. Therefore,
    $$\left\langle \frac{q+1}{2},\, q+1 - \frac{m}{2},\, \ldots,\, q+1
      \right\rangle \subseteq H(P_{\infty}^1).$$
    The equality of both numerical semigroups follows from
    Proposition~\ref{genus-Pinf}.
\end{proof}

\begin{proposition}\label{prop:num-inf-odd}
    Let $m$ be odd. We have
    $$H(P_{\infty}^i) =
      \left\langle q+1-m,\, q+1-m+2,\, \ldots,\, q,\, q+1 \right\rangle.$$
\end{proposition}

\begin{proof}
    The first part of this proof is very similar to that of
    Lemma~\ref{Pinf-m-even}. However, we need to adjust the local parameter at
    $P_{\infty}^1$: since $v_{P_{\infty}^1}(x) = -m$ and $p \nmid m$, we may
    apply Hensel's lemma in the completed local ring at $P_{\infty}^1$ to write
    $x = \pi^{-m}$ for some local parameter $\pi$. We then obtain
    $$y = \beta_1\pi^{-2} + \gamma_1\pi^{q+1} + \textrm{h.o.t.},$$
    with $\beta_1^{q+1} = -1$.

    The function $\pi$ is a local parameter at each $P_{\infty}^i$, and the
    coefficients $\beta_i$ are related by multiplication by a primitive
    $(q+1)$-th root of unity (as can also be seen from the covering map
    $\mathcal{H} \rightarrow \mathcal{X}$). In particular, the elements
    $\beta_i^m$ are pairwise distinct. We define the polynomial $F$ whose roots
    are $\bigl(\beta_i^{-m}\bigr)_i$.

    Proceeding as in the proof of Lemma~\ref{Pinf-m-even}, we obtain a function
    $f$ such that
    $$\left(\frac{f}{x^{(q+1)/m-1}}\right)
      = m \sum_{i=2}^{(q+1)/m} P_{\infty}^i - (q+1-m)P_{\infty}^1.$$
    The remaining generators of the semigroup are obtained by multiplying by
    suitable powers of $y$, and the equality of both numerical semigroups
    follows from Proposition~\ref{genus-Pinf}.
\end{proof}

\begin{remark}
    \rm{Based on Propositions~\ref{prop:num-zeros}, \ref{prop:num-inf-even},
    and~\ref{prop:num-inf-odd}, the places of $\mathcal{O}$ share the same
    Weierstrass semigroup if and only if $m = (q+1)/2$ is even.}
\end{remark}

This is a key difference between our work and the case $m = 3$ studied
in~\cite{BMV}. In that paper, the authors showed that $\mathcal{O}$ is always
an orbit under the action of $\textrm{Aut}(\mathcal{X}_{3,q})$, but in our
case $\mathcal{O}$ can only be an orbit if $m = (q+1)/2$ and $m$ is even (which
is equivalent to $q \equiv 3 \pmod{4}$). In fact, in this case the curve
$\mathcal{X}_{(q+1)/2,q}$ admits more automorphisms than those found in
Proposition~\ref{automorfismos}.

\begin{proposition}\label{prop:bigger-aut}
    Let $q \ge 7$ with $q \equiv 3 \pmod{4}$, and let
    $\delta_1, \delta_2, \delta_3 \in \mathbb{F}_{q^2}^*$ be 
such that
$$
(\delta_3)^{q+1} = \tfrac{1}{16},\quad
(\delta_2)^4 = -\tfrac{1}{16},\quad
\delta_1 = -4^{-1} (\delta_2)^{-1}
$$
Then there is an automorphism $\theta_4 \in \Aut(\mathcal{X}_{(q+1)/2,q})$ given by
    \begin{equation}\label{theta4}
    \theta_4^{\delta_1,\delta_2,\delta_3}(x,y) =
        \theta_4(x, y) \coloneqq
        \left(\frac{\delta_1 x}{y^{(q+1)/4}}+\frac{\delta_2 y^{(q+1)/4}}{x},\,
              \frac{\delta_3}{y}\right).
    \end{equation}
Moreover, if $(\delta_3)^{(q+1)/2} = \frac{1}{4}$ and if $G \subseteq \Aut(\mathcal{X}_{(q+1)/2,q})$ is the subgroup generated by $A \cup \{\theta_4\}$, where
    $$A \coloneqq \{\theta_{\gamma,\delta}(x,y) = (\gamma x, \delta y)
            \mid \gamma^2 = \delta^{q+1} = 1\}$$
then $G$ has order $8(q+1)$.
\end{proposition}

\begin{proof}
    We first verify that $\theta_4$ is an automorphism of the curve. Note that
    $$\delta_1^4 = \delta_2^4 = -\tfrac{1}{16}, \quad
      \delta_3^{q+1} = \tfrac{1}{16}, \quad
      6\delta_1^2\delta_2^2 + 2\delta_1\delta_2 = -\tfrac{2}{16}, \quad
      4\delta_1\delta_2 + 1 = 0.$$
    Therefore,
    \begin{align*}
        &\left(\frac{\delta_3}{y}\right)^{q+1}
         + \left(\frac{\delta_1x}{y^{(q+1)/4}}
                 + \frac{\delta_2y^{(q+1)/4}}{x}\right)^4
         + \left(\frac{\delta_1x}{y^{(q+1)/4}}
                 + \frac{\delta_2y^{(q+1)/4}}{x}\right)^2 \\
        &= \frac{\delta_3^{q+1}x^4 + \delta_1^4x^8
                 + (4\delta_1^3\delta_2 + \delta_1^2)x^6y^{\frac{q+1}{2}}
                 + (6\delta_1^2\delta_2^2 + 2\delta_1\delta_2)x^4y^{q+1}
                 + (4\delta_1\delta_2^3 + \delta_2^2)x^2y^{\frac{3(q+1)}{2}}
                 + \delta_2^4y^{2(q+1)}}
               {x^4y^{q+1}} \\
        &= \frac{y^{2(q+1)}+2x^4y^{q+1}+x^8-x^4}{16x^4y^{q+1}}
         = \frac{y^{2(q+1)}+2x^4y^{q+1}+2x^2y^{q+1}+x^8+2x^6+x^4}
               {16x^4y^{q+1}} \\
        &= \frac{(y^{q+1}+x^4+x^2)^2}{16x^4y^{q+1}},
    \end{align*}
    which shows that $\theta_4 \in \Aut(\mathcal{X}_{(q+1)/2,q})$.

    Since $2^{-4} = \delta_3^{q+1} = (\delta_3^{(q+1)/4})^4$, the value $2\delta_3^{(q+1)/4}$ is a fourth root of unity. We may verify that:
    \begin{itemize}
    \item
    When $2\delta_3^{(q+1)/4} \in \{1,-1\}$, then $\theta_4\circ \theta_4 = \theta_{-2\delta_3^{(q+1)/4}, 1} \circ \theta_2 \notin A$ (in the notation of Proposition~\ref{automorfismos}).
    \item
    When $2\delta_3^{(q+1)/4} \notin \{1,-1\}$, then $8\delta_3^{(q+1)/4}\delta_2^2 \in \{1,-1\}$ and $\theta_4\circ \theta_4 = \theta_{8\delta^{(q+1)/4}\delta_2^2, 1} \in A$.
    \end{itemize}
In particular, $(\theta_4)^4 = \mathrm{id}$.

    The same formulas for $\theta_4 \circ \theta_4$ can be applied to
    $$
    \theta_{\gamma,\delta} \circ \theta_4^{\delta_1,\delta_2,\delta_3} = \theta_4^{\gamma\delta_1,\gamma\delta_2,\delta\delta_3}
    $$
    in place of $\theta_4 = \theta_4^{\delta_1,\delta_2,\delta_3}$.

So, we have, in the case $\delta_3^{(q+1)/2}=\frac{1}{4}$,
\begin{align*}
\theta_{4}\circ\theta_{\gamma,\delta} & =\theta_{\gamma,\delta}^{-1}\circ(\theta_{\gamma,\delta}\circ\theta_{4})^{2}\circ(\theta_{4})^{-1}\\
 & \in\theta_{\gamma,\delta}^{-1}\circ A\circ\{(\theta_4)^2,\mathrm{id}\}\circ(\theta_{4})^{-1}\\
 & =A\circ\{\theta_{4},(\theta_{4})^{3}\},
\end{align*}
and so $\left\langle \theta_{4}\right\rangle \circ A\subseteq A \circ \left\langle \theta_{4}\right\rangle $;
so $A\circ \left\langle \theta_{4}\right\rangle $ is a subgroup, and has
exactly $8(q+1)$ elements, because $A\cap\left\langle \theta_{4}\right\rangle =\{\mathrm{id}\}$.

\end{proof}

One can verify that $\theta_4(\mathcal{O}_m) \subseteq \mathcal{O}_{\infty}$,
so in the case where $m = (q+1)/2$ is even, $\mathcal{O}$ can indeed be an
orbit under the action of $\textrm{Aut}(\mathcal{X}_{(q+1)/2,q})$. To confirm
this, it remains to show that the Weierstrass semigroups at the rational places
outside $\mathcal{O}$ differ from those of $\mathcal{O}$. This will be done in
the following section.

\section{\texorpdfstring
{The Weierstrass semigroups at the places outside of $\mathcal{O}$}
{The Weierstrass semigroups at the places outside of O}}
\label{sec:weierstrass-other}

Although $\mathcal{O}$ is not always an orbit under the action of the
automorphism group of $\mathcal{X}_{m,q}$, the Weierstrass semigroups at the
places outside of $\mathcal{O}$ differ from those of $\mathcal{O}$. This is
analogous to the result obtained in \cite{BMV} for the case $m = 3$. Following
their approach, we extend the argument to the curves $\mathcal{X}_{m,q}$ in
order to obtain the Weierstrass semigroups at the remaining rational places.
As will be seen in the following subsections, this depends heavily on the
values of $m$ and on the characteristic of $\mathbb{F}_q$. We first present
the case $m = 4$ in detail, and then discuss how to extend this to the general
case.

\subsection{\texorpdfstring{The case $m = 4$}{The case m = 4}}

Let $P_{(a,b)}$, $ab \neq 0$, be a rational place not in $\mathcal{O}$. Let
$Q_{(A,B)}$, with $A^{m} = a$ and $AB = b$, be a place of the Hermitian
function field lying over $P_{(a,b)}$. Then, as observed in
Section~\ref{sec:results}, $e\!\left(Q_{(A,B)} \mid P_{(a,b)}\right) = 1$.
Let $T \coloneqq (u-A)/A$ be a local parameter at $Q_{(A,B)}$. Then
$$\frac{x-a}{a} = \frac{u^{m}-A^{m}}{A^{m}}
  = (T+1)^{m} - 1
  = \binom{m}{1}T + \binom{m}{2}T^{2} + \cdots + \binom{m}{m}T^{m},$$
\begin{align*}
    \frac{y-b}{b}
    &= \frac{uv - AB}{AB}
     = \frac{(u-A+A)(v-B+B) - AB}{AB} \\
    &= \frac{(u-A)(v-B) + A(v-B) + B(u-A)}{AB}
     = T\frac{v-B}{B} + \frac{v-B}{B} + T.
\end{align*}
From the equation of the curve,
\begin{align*}
    0 &= u^{q+1} + v^{q+1} + 1 \\
      &= (u-A+A)^{q+1} + (v-B+B)^{q+1} + 1 \\
      &= (u-A)^{q+1} + A^{q}(u-A) + A(u-A)^{q} + A^{q+1} \\
      &\quad + (v-B)^{q+1} + B^{q}(v-B) + B(v-B)^{q} + B^{q+1} + 1,
\end{align*}
which gives
\begin{align*}
    0 &= (u-A)^{q+1} + (v-B)^{q+1} + A^{q}(u-A) + B^{q}(v-B)
         + A(u-A)^{q} + B(v-B)^{q} \\
      &= B^{q}(v-B) + A^{q+1}T + O\!\left(T^{q}\right),
\end{align*}
so
\begin{align*}
    \frac{y-b}{b}
    &= (T+1)\frac{v-B}{B} + T \\
    &= (T+1)\left(-\left(\frac{A}{B}\right)^{q+1}T + O\!\left(T^{q}\right)\right) + T \\
    &= \left(1 - \left(\frac{A}{B}\right)^{q+1}\right)T
       - \left(\frac{A}{B}\right)^{q+1}T^{2} + O\!\left(T^{q}\right) \\
    &= \left(\frac{2a^{(q+1)/m}+1}{1+a^{(q+1)/m}}\right)T
       + \frac{a^{(q+1)/m}}{1+a^{(q+1)/m}}\,T^{2} + O\!\left(T^{q}\right),
\end{align*}
where $O(T^{q})$ denotes an element with $Q_{(A,B)}$-valuation at least $q$.

Defining
\[
    \alpha \coloneqq \frac{a^{(q+1)/m}}{1+a^{(q+1)/m}}
    = \frac{A^{q+1}}{1+A^{q+1}} = \frac{A^{q+1}}{-B^{q+1}},
\]
we have
\[
    X \coloneqq \frac{x-a}{a} = (T+1)^{m} - 1
    = \binom{m}{1}T + \binom{m}{2}T^{2} + \cdots + \binom{m}{m}T^{m},
\]
\[
    Y \coloneqq \frac{y-b}{b} = (1+\alpha)T + T^{2} + O\!\left(T^{q}\right).
\]
We also define
\begin{align*}
    t_0 &\coloneqq (1+\alpha)\cdot\frac{x-a}{a} - m\cdot\frac{y-b}{b} \\
        &= (1+\alpha)\sum_{i=2}^{m}\binom{m}{i}T^{i}
           - m\alpha T^{2} + O\!\left(T^{q}\right) \\
        &= \left(\binom{m}{2} + \left(\binom{m}{2}-m\right)\alpha\right)T^{2}
           + (1+\alpha)\sum_{i=3}^{m}\binom{m}{i}T^{i}
           + O\!\left(T^{q}\right).
\end{align*}

\begin{defn}
    Let $\iota \in \overline{K}$ be a square root of $-1$. Since $4 \mid q+1$,
    we have $\iota \in K$ and $\iota \neq -\iota$. We define, for each
    $k \in \mathbb{N}$:
    \[
    \mathcal{P}_{k}(s) \coloneqq
    \frac{\iota(s+\iota)^{4k} - \iota(s-\iota)^{4k}}{2\cdot(s - s^{2})}
    \quad \text{and} \quad
    \mathcal{Q}_{k}(s) \coloneqq
    \frac{(1-\iota)(s+\iota)^{4k-1} + (1+\iota)(s-\iota)^{4k-1}}{2\cdot(s-1)}.
    \]
\end{defn}

One can verify that, for each $k \in \mathbb{N}$, $\mathcal{P}_{k}(s)$ and
$\mathcal{Q}_{k}(s)$ are polynomials in $\mathbb{F}_{p}[s]$ having no common
roots. Our goal is to find, for each $\alpha$, the values $k \in \mathbb{N}$
for which there exists
\begin{equation}
    f_{k} \in
    \left\langle X^{i} \cdot Y^{j} \right\rangle_{i,j \in \mathbb{N}}^{i/2+j \le k}
    \subseteq \mathcal{L}\!\left(k \cdot (X)_{\infty}\right)
    \label{eq:fk L k Xinf}
\end{equation}
with $T$-expansion
\begin{equation}
    f_{k} = \mathcal{P}_{k}(\alpha) \cdot T^{4k-1}
           + \mathcal{Q}_{k}(\alpha) \cdot T^{4k}
           + O\!\left(T^{q}\right)
    \label{eq:fk Pk Qk}
\end{equation}
and $4k < q$. We will use these elements to describe the Weierstrass semigroup
at $P_{(a,b)}$.

We obtain these elements by successively cancelling terms of equal
$T$-valuation. We have $\mathcal{P}_1(s) = 4(s+1)$ and
$\mathcal{Q}_1(s) = s^2 + 4s + 1$. Since $\alpha \neq 1$, we may define
\[
    f_{1} \coloneqq
    \frac{(2\alpha+6) \cdot Y^{2} - (\alpha+1)^{2} t_{0}}{\alpha - 1},
\]
which satisfies \eqref{eq:fk L k Xinf} and \eqref{eq:fk Pk Qk}. If
$\mathcal{Q}_{1}(\alpha) = \alpha^{2} + 4\alpha + 1 \neq 0$, we may define
\begin{align*}
    \mathcal{Q}_{1}(\alpha)^{2} \cdot f_{2}
    \coloneqq{} &
    -2^{8}(\alpha+1)^{4} \cdot f_{1}
    + 2^{10}(\alpha+1)^{2} \cdot Y^{3}
    + 2^{6}(\alpha+1)^{2}(\alpha^{2}-8\alpha+1) \cdot f_{1}Y \\
    &- 2^{4}\mathcal{Q}_{1}(\alpha)\mathcal{Q}_{1}(-\alpha)\,f_{1}Y^{2}
    + \left(\mathcal{Q}_{2}(\alpha) + 16\alpha^{2}\mathcal{Q}_{1}(-\alpha)\right)f_{1}^{2},
\end{align*}
and $f_2$ satisfies \eqref{eq:fk L k Xinf} and \eqref{eq:fk Pk Qk}. Still
assuming $\mathcal{Q}_{1}(\alpha) \neq 0$, we obtain $f_{3}$ as a
$\mathbb{Z}[\alpha]$-linear combination of
$f_{2},\, f_{1}^{2}Y,\, f_{1}^{2}Y^{2},\, f_{1}^{3},\, f_{1}f_{2}$,
divided by $\mathcal{Q}_{1}(\alpha)^{4}$.

Now, for $k \geq 4$, given $f_{1}, \ldots, f_{k-1}$ as in
\eqref{eq:fk L k Xinf} and \eqref{eq:fk Pk Qk}, and assuming
$\mathcal{P}_{k-3}(\alpha) \neq 0$, $\alpha + 1 \neq 0$, and
$\alpha^{2} + 1 \neq 0$, we may define
\[
    f_{k} \coloneqq
    \frac{\mathcal{P}_{k-2}(\alpha)\mathcal{P}_{2}(\alpha) \cdot f_{k-1}f_{1}
          + \mathcal{P}_{k-1}(\alpha)\mathcal{P}_{1}(\alpha) \cdot f_{k-2}f_{2}}
         {2^{2}(\alpha+1)(\alpha^{2}+1)^{3}\mathcal{P}_{k-3}(\alpha)},
\]
which also satisfies \eqref{eq:fk L k Xinf} and \eqref{eq:fk Pk Qk}.

For most places $P$, there exists a minimal value of $k$ for which
$\mathcal{P}_{k}(\alpha) = 0$. Each such value yields a different Weierstrass
semigroup.
\begin{lemma}\label{lem:I min Pk}
    Suppose that $\alpha \notin \{\iota, -\iota\}$. Then the set
    $\{k \ge 1 \mid \mathcal{P}_{k}(\alpha) = 0\}$ is nonempty, and we denote
    its minimum by
    \[
        I \coloneqq \min\left\{ k \ge 1 \mid \mathcal{P}_{k}(\alpha) = 0 \right\}.
    \]
    Moreover, $I$ is a divisor of $\frac{q+1}{4}$ (in particular,
    $4I \le q+1$).
\end{lemma}

\begin{proof}
    We have $\mathcal{P}_{k}(\alpha) = 0$ if and only if
    $\left(\frac{\alpha+\iota}{\alpha-\iota}\right)^{4k} = 1$, thus it suffices
    to show that $\left(\frac{\alpha+\iota}{\alpha-\iota}\right)^{q+1} = 1$.
    In fact,
    \[
        \left(\frac{\alpha+\iota}{\alpha-\iota}\right)^{q}
        = \frac{\alpha^{q}+\iota^{-1+4\frac{q+1}{4}}}
               {\alpha^{q}-\iota^{-1+4\frac{q+1}{4}}}
        = \frac{\alpha^{q}-\iota}{\alpha^{q}+\iota},
    \]
    thus it suffices to show that $\alpha^{q} = \alpha$. Since
    $\alpha = \frac{a^{(q+1)/4}}{1+a^{(q+1)/4}}$, it suffices to show that
    $a^{(q+1)/4} \in \mathbb{F}_{q}$.

    Since $(a^{(q+1)/4})^{q-1} = a^{(q^{2}-1)/4} \neq 0$ is a fourth root of
    $a^{q^{2}-1} = 1$, we have
    $a^{(q^{2}-1)/4} \in \{1, -1, \iota, -\iota\}$.
    Given that $b^{q+1} + a^{2(q+1)/4} + a^{(q+1)/4} = 0$ and
    $b \in \mathbb{F}_{q^{2}} \smallsetminus \{0\}$, we have
    \[
        \left(a^{2\frac{q+1}{4}} + a^{\frac{q+1}{4}}\right)^{q-1}
        = \left(-b^{q+1}\right)^{q-1}
        = (-1)^{q-1} b^{q^{2}-1} = 1,
    \]
    so $a^{2(q+1)/4} + a^{(q+1)/4} \in \mathbb{F}_{q}$. Thus,
    \[
        a^{2\frac{q+1}{4}} + a^{\frac{q+1}{4}}
        = \left(a^{2\frac{q+1}{4}} + a^{\frac{q+1}{4}}\right)^{q}
        = a^{2\frac{q^{2}+q}{4}} + a^{\frac{q^{2}+q}{4}},
    \]
    that is,
    \[
        \left(a^{\frac{q+1}{4}}\right)^{2} + a^{\frac{q+1}{4}}
        = \left(a^{\frac{q^{2}-1}{4}} a^{\frac{q+1}{4}}\right)^{2}
          + a^{\frac{q^{2}-1}{4}} a^{\frac{q+1}{4}},
    \]
    \begin{align*}
        a^{\frac{q+1}{4}} - a^{\frac{q^{2}-1}{4}} a^{\frac{q+1}{4}}
        &= \left(a^{\frac{q^{2}-1}{4}} a^{\frac{q+1}{4}}\right)^{2}
           - \left(a^{\frac{q+1}{4}}\right)^{2} \\
        &= \left(a^{\frac{q^{2}-1}{4}} a^{\frac{q+1}{4}} + a^{\frac{q+1}{4}}\right)
           \left(a^{\frac{q^{2}-1}{4}} a^{\frac{q+1}{4}} - a^{\frac{q+1}{4}}\right).
    \end{align*}

    Suppose, for a contradiction, that $a^{(q^{2}-1)/4} \neq 1$. 
    
    Then
    $a^{\frac{q^{2}-1}{4}} a^{\frac{q+1}{4}} + a^{\frac{q+1}{4}} = -1$
    and $a^{\frac{q^{2}-1}{4}} + 1 = -a^{-\frac{q+1}{4}}$. Also,
    \begin{align*}
        0 \neq \alpha^{2} + 1
        &= \left(\frac{a^{\frac{q+1}{4}}}{1+a^{\frac{q+1}{4}}}\right)^{2} + 1
         = \frac{a^{2\frac{q+1}{4}}
                 + \left(1 + 2a^{\frac{q+1}{4}} + a^{2\frac{q+1}{4}}\right)}
                {1 + 2a^{\frac{q+1}{4}} + a^{2\frac{q+1}{4}}},
    \end{align*}
    so $1 + 2(a^{\frac{q+1}{2}} + a^{\frac{q+1}{4}}) \neq 0$. Multiplying by
    $a^{-(q+1)/2}$, we obtain
    \begin{align*}
        0 &\neq a^{-\frac{q+1}{2}} + 2 + 2a^{-\frac{q+1}{4}} \\
          &= \left(a^{\frac{q^{2}-1}{4}} + 1\right)^{2}
             + 2 - 2\left(a^{\frac{q^{2}-1}{4}} + 1\right) \\
          &= \left(a^{\frac{q^{2}-1}{2}} + 2a^{\frac{q^{2}-1}{4}} + 1\right)
             + 2 - \left(2a^{\frac{q^{2}-1}{4}} + 2\right) \\
          &= a^{\frac{q^{2}-1}{2}} + 1,
    \end{align*}
    so $a^{\frac{q^{2}-1}{4}} \notin \{\iota, -\iota\}$, leaving only the
    possibility $a^{\frac{q^{2}-1}{4}} = -1$. But then
    $a^{\frac{q^{2}-1}{4}} + 1 = 0$, which gives
    $-a^{-\frac{q+1}{4}} = 0$, contradicting $a \neq 0$.
\end{proof}

We now determine the Weierstrass semigroups at every rational point of 
$\mathcal{X}_{m,q}$. Let $P_{(a,b)} \in \mathcal{X}_{m,q}(\mathbb{F}_{q^2})$ 
with $ab \neq 0$ be a rational place outside $\mathcal{O}$, and let 
$P_{(\overline{a},0)}$ be a rational place of $\mathcal{O}_m$, so that 
$\overline{a}^{\frac{q+1}{m}} + 1 = 0$. For $k \in \mathbb{N}$ and 
$h_k \in \mathcal{L}(k \cdot (x)_\infty)$, define
\begin{equation}\label{div-Gk}
    G_k \coloneqq \frac{h_k \cdot f_{P_{(\overline{a},0)}}^k}
    {f_{P_{(a,b)}}^k \cdot (x - \overline{a})^k},
\end{equation}
where $f_{P_{(a,b)}}$ and $f_{P_{(\overline{a},0)}}$ are chosen via the 
fundamental relations
\[
    (f_{P_{(a,b)}}) = (q+1)P_{(a,b)} - (q+1)P_\infty^1,
\]
\[
    (f_{P_{(\overline{a},0)}}) = (q+1)P_{(\overline{a},0)} - (q+1)P_\infty^1.
\]
Since $(x - \overline{a}) = (q+1)P_{(\overline{a},0)} - (x)_\infty$, a 
straightforward divisor computation gives
\[
    (G_k) = (h_k) - k(q+1)P_{(a,b)} + k(x)_\infty,
\]
and in particular,
\[
    k(q+1) - v_{P_{(a,b)}}(h_k) \in H(P_{(a,b)}).
\]

\begin{proposition}\label{conj-Yuri-4}
    Suppose $\alpha \notin \{\iota, -\iota\}$ and 
    $\mathcal{Q}_1(\alpha) = \alpha^2 + 4\alpha + 1 \neq 0$. 
    Then there exists a divisor $I$ of $\frac{q+1}{4}$ such that 
    $H(P_{(a,b)})$ is generated by
    \[
        \{q,\, q+1,\, q-1\} \cup 
        \Bigl\{\, k(q+1) - (4k-1) - [k = I] 
        \;\Big|\; k = 1, \ldots, \min\!\Bigl(I,\, \tfrac{q+1}{4} - 1\Bigr) \Bigr\},
    \]
    where $[k = I]$ denotes the Iverson bracket, equal to $1$ if $k = I$ 
    and $0$ otherwise.
\end{proposition}

\begin{proof}
    Let $I$ be as in Lemma~\ref{lem:I min Pk}. By the discussion preceding 
    this proposition, for each 
    $k \in \{1, \ldots, \min(I, \frac{q+1}{4} - 1)\}$ 
    there exist functions $f_k$ satisfying \eqref{eq:fk L k Xinf} 
    and~\eqref{eq:fk Pk Qk} whose $T$-valuation equals 
    $(4k-1) + [k = I]$ (note that $\mathcal{P}_k(\alpha) = 0$ 
    implies $\mathcal{Q}_k(\alpha) \neq 0$, and $4k < q$). 
    Applying~\eqref{div-Gk} and the remark following it yields
    \[
        k(q+1) - (4k-1) - [k=I] \;\in\; H(P_{(a,b)}).
    \]
    It is known that $q, q+1 \in H(P)$ for every rational place $P$ of a 
    maximal curve (see, e.g., \cite[Proposition~1]{FM}). Setting $h_k = Y$ 
    when $\alpha = -1$, or $h_k = Y^2$ when $\alpha \neq -1$, 
    in~\eqref{div-Gk} shows that $q - 1 \in H(P_{(a,b)})$. It remains to 
    verify that the stated generators produce a numerical semigroup of genus 
    at most $g(\mathcal{X}_{m,q})$; this is established in 
    Lemma~\ref{lem:bound genus semigr m4}.
\end{proof}

\begin{remark}\label{rmk:m4}
   \rm{ Since $\mathcal{P}_1(s) = 4(s+1)$, we have $I = 1$ whenever $\alpha = -1$. 
    In particular,
    \[
        H(P_{(a,b)}) = H(P_{(\overline{a},0)})
    \]
    for $2a^{(q+1)/4} + 1 = 0$ and $\overline{a}^{(q+1)/4} + 1 = 0$.}
\end{remark}

\begin{proposition}\label{prop:alpha-iota}
    Suppose $\alpha \in \{\iota, -\iota\}$. Then $H(P_{(a,b)})$ is generated by
    \[
        \{q,\, q+1,\, q-1\} \cup 
        \Bigl\{\, k(q+1) - (4k-1) \;\Big|\; 
        k = 1, \ldots, \tfrac{q+1}{4} - 1 \Bigr\}.
    \]
\end{proposition}

\begin{proof}
    A direct computation gives
    \[
        \mathcal{P}_k(\alpha) 
        = \frac{\alpha(2\alpha)^{4k}}{2(\alpha+1)} 
        = 2^{4k-2}(\alpha+1)
        \qquad \text{and} \qquad
        \mathcal{Q}_k(\alpha) 
        = \frac{(1-\alpha)(2\alpha)^{4k-1}}{2(\alpha-1)} 
        = 2^{4k-2}\alpha.
    \]
    As before, there exist functions $f_1, f_2, \ldots$ satisfying 
    \eqref{eq:fk L k Xinf} and~\eqref{eq:fk Pk Qk}, though their 
    construction differs in this case. Write $f_k \coloneqq 2^{4k-2} g_k$, 
    where $f_1, f_2, f_3$ are as previously defined, and for $k \geq 3$ 
    the functions $g_k$ satisfy the recurrence
    \[
        g_k \coloneqq -4\, g_{k-1} 
        - 2\alpha\, g_{k-2} g_1 Y 
        + 2\alpha\, g_{k-2} g_1 Y^2 
        + 5\, g_{k-2} g_1^2 
        - 4\alpha\, g_{k-1} g_1.
    \]
    For each $k \in \{1, \ldots, \frac{q+1}{4} - 1\}$, the function $g_k$ 
    lies in $\mathcal{L}(k \cdot (X)_\infty)$ and has $T$-valuation equal 
    to $4k - 1$, so that
    \[
        k(q+1) - (4k-1) \;\in\; H(P_{(a,b)}).
    \]
    Applying~\eqref{div-Gk} with $h_k \in \{1, X, Y^2\}$ yields 
    $\{q+1, q, q-1\} \subseteq H(P_{(a,b)})$. To conclude that the 
    stated generators produce the full Weierstrass semigroup, we apply 
    Lemma~\ref{lem:bound genus semigr m4} with $I \coloneqq \frac{q+1}{4}$.
\end{proof}
   
	\begin{proposition}\label{prop:Q1-zero}
    Let $P_{(a,b)}$ be a rational place with 
    $\mathcal{Q}_1(\alpha) = \alpha^2 + 4\alpha + 1 = 0$. Then 
    $H(P_{(a,b)})$ is generated by
    \[
        \{q+1,\, q,\, q-1,\, q-2\} 
        \cup \{2(q+1) - 7\} 
        \cup \{3(q+1) - 12\}.
    \]
\end{proposition}

\begin{proof}
    Since $(\alpha+2)^2 = 3$ and $\alpha \neq 1$, the characteristic of 
    $\mathbb{F}_q$ is not $3$. Moreover, $\alpha \neq -1$, so the elements 
    $1, Y, Y^2 \in \mathcal{L}(1 \cdot (X)_\infty)$ have $T$-valuations 
    $0$, $1$, and $2$, respectively, giving $q+1, q, q-1 \in H(P_{(a,b)})$. 
    Furthermore,
    \[
        f_1 = 4(\alpha+1)T^3 + O(T^q) 
        \;\in\; \mathcal{L}(1 \cdot (X)_\infty),
    \]
    so $q - 2 \in H(P_{(a,b)})$.

    From the proof of Lemma~\ref{lem:I min Pk} and the fact that 
    $\alpha^2 + 1 = -4\alpha \neq 0$, we have $\alpha \in \mathbb{F}_q$. 
    In particular, $3$ is a square in $\mathbb{F}_q$ and $q \neq 7$. Define
    \begin{align*}
        g_{2,7} &\coloneqq \frac{1}{12}\Bigl[
            (-56\alpha - 208)Y^2 
            + (-58\alpha - 218)Y^3 
            + (-33\alpha - 123)Y^4 \\
            &\qquad\qquad 
            + (10\alpha + 38)XY 
            + (15\alpha + 57)XY^2
        \Bigr] \;\in\; \mathcal{L}(2 \cdot (X)_\infty),
    \end{align*}
    so that $g_{2,7} = T^7 - \frac{\alpha+1}{2}T^8 + O(T^q)$. 
    Since $q > 7$, the function $g_{2,7}$ has $T$-valuation $7$, 
    and hence $2(q+1) - 7 \in H(P_{(a,b)})$.

    Next, let $g_{3,12}$ be an appropriate $\mathbb{Z}[\alpha]$-linear 
    combination of $Y, \ldots, Y^6, X, XY, XY^3, XY^4$, divided by $3$, 
    such that $g_{3,12} \in \mathcal{L}(3 \cdot (X)_\infty)$ and 
    $g_{3,12} = T^{12} + O(T^q)$. If $q \neq 11$, then $q > 12$ and 
    $g_{3,12}$ has $T$-valuation $12$, giving 
    $3(q+1) - 12 \in H(P_{(a,b)})$. If $q = 11$, then 
    $3(q+1) - 12 = 24 = 2(q+1) \in H(P_{(a,b)})$, so the conclusion 
    holds in either case.

    It remains to apply Lemma~\ref{lem:bound genus semigr m4} with $I = 3$, 
    for which we need $3 \mid (q+1)/4$, i.e., $q \equiv -1 \pmod{12}$. 
    Since $\alpha$ satisfies a quadratic equation over $\mathbb{F}_p$ and 
    $q$ is not a perfect square (as $q \equiv -1 \pmod 4$), we have 
    $\alpha \in \mathbb{F}_q \cap \mathbb{F}_{p^2} = \mathbb{F}_p$. 
    Thus $3 = (\alpha+2)^2$ is a quadratic residue modulo $p \neq 2, 3$, 
    and the law of quadratic reciprocity gives
    \[
        (-1)^{\frac{3-1}{2} \cdot \frac{p-1}{2}} 
        = \left(\frac{3}{p}\right)\left(\frac{p}{3}\right) 
        = 1 \cdot (-1)^{[p \equiv -1\,(\mathrm{mod}\,3)]},
    \]
    so $[p \equiv -1 \pmod{4}] = [p \equiv -1 \pmod{3}]$, meaning 
    $p \equiv \pm 1 \pmod{12}$. Since $q \equiv -1 \pmod{4}$, we 
    conclude $p \equiv q \equiv -1 \pmod{12}$.
\end{proof}

	\begin{lemma}\label{lem:bound genus semigr m4}
    Let $I \in \{1, \ldots, \frac{q+1}{4}\}$ be a divisor of $\frac{q+1}{4}$. 
    The genus of the numerical semigroup
    \[
        Q(I) \coloneqq \langle q-1,\, q,\, q+1 \rangle 
        + \Bigl\langle k(q+1) - (4k-1) - [k=I] \;\Big|\; 
        k = 1, \ldots, \min\!\Bigl(I,\, \tfrac{q+1}{4} - 1\Bigr) \Bigr\rangle
    \]
    is at most $g = g(\mathcal{X}_{4,q})$.
\end{lemma}

\begin{proof}
    If $(q, I) = (7, 2)$, then $Q(I) = \langle 5, 6, 7, 8 \rangle$ has 
    genus $5 = g$. Henceforth assume $(q, I) \neq (7, 2)$. Since
    \[
        k(q+1) - (4k-1) - [k=I] = k(q-3) + 1 - [k=I],
    \]
    the smallest positive element of $Q(I)$ is 
    $a \coloneqq q - 2 - [I=1]$ (noting that $\frac{q+1}{4} > 1$).

    We construct $f \colon \{1, \ldots, a-1\} \to Q(I)$ such that 
    $f(v) \equiv v \pmod{a}$ for each $v$ and 
    $\sum_{v=1}^{a-1} \lfloor f(v)/a \rfloor = g$.

    \medskip
    \noindent\textbf{Case $I > 1$} (so $a = q-2$ and $\frac{q+1}{4} > 2$).
    Let $v \in \{1, \ldots, a-1\}$.

    \begin{itemize}
        \item \textbf{Subcase $1 \le v \le 3\frac{q+1}{4} - 3$}. 
        Write $v = \lceil v/3 \rceil \cdot 3 - w$ with $w \in \{0,1,2\}$, 
        and define
        \begin{align*}
            f(v) &\coloneqq \Bigl(\lceil \tfrac{v}{3} \rceil - 2\Bigr)(q+1) 
            + \bigl(q+1 - [w \ge 1]\bigr) 
            + \bigl(q+1 - [w \ge 2]\bigr) \\
            &= \lceil \tfrac{v}{3} \rceil (q+1) - w 
            = v + \lceil \tfrac{v}{3} \rceil \cdot a,
        \end{align*}
        which lies in $Q(I)$ when $\lceil v/3 \rceil \ge 2$. 
        For $\lceil v/3 \rceil = 1$, we have 
        $f(v) \in \{q-1, q, q+1\} \subseteq Q(I)$.

        \item \textbf{Subcase $a - \frac{q+1}{4} + 1 \le v \le a-1$} 
        (equivalently, $1 \le a - v \le \frac{q+1}{4} - 1$). 
        Let $b_k \coloneqq k(q-3) + 1 - [k=I]$ denote the generators above 
        (so $a = b_1$). These satisfy 
        $b_k \equiv -k + 1 - [k=I] \pmod{a}$. Set
        \[
            s \coloneqq \bigl(-b_{I - [I=(q+1)/4]}\bigr) \bmod a 
            = I - 2\bigl[I = \tfrac{q+1}{4}\bigr] > 0,
        \]
        and write $a - v = sQ + R$ with $Q \in \mathbb{N}$ and 
        $R \in \{0, \ldots, s-1\}$. Define
        \[
            f(v) \coloneqq Q\, b_{I-[I=(q+1)/4]} + 
            \begin{cases}
                0 & (R = 0), \\
                b_{I-1} + b_2 & (I < \tfrac{q+1}{4},\ R = I-1), \\
                b_{R+1} & (\text{otherwise}).
            \end{cases}
        \]
        When $I < \frac{q+1}{4}$ and $R = I-1$, we have
        \[
            b_{I-1} + b_2 
            = (I-1)(q-3) + 1 + 2(q-3) + 1 
            = (R+1)(q-3) + 1 + (q-2),
        \]
        so in this subcase $f(v) = a\bigl(a - v + [R=I-1] - [R=0]\bigr) + v$.
        When $I = \frac{q+1}{4}$ (and $s = I-2$), one checks that 
        $Q = [a-v \in \frac{q+1}{4} - \{1,2\}]$ and 
        $[R=0] = [a-v = \frac{q+1}{4}-2]$, giving
        \[
            f(v) = a\bigl(a - v + Q - [R=0]\bigr) + v.
        \]
    \end{itemize}

    Summing $\lfloor f(v)/a \rfloor$ over $v \in \{1,\ldots,a-1\}$ 
    and using the formulas above, one computes:
    \begin{align*}
        \sum_{v=1}^{a-1} \Bigl\lfloor \frac{f(v)}{a} \Bigr\rfloor
        &= \sum_{v=1}^{3((q+1)/4-1)} \!\!\Bigl\lceil \tfrac{v}{3} \Bigr\rceil
        + \sum_{w=1}^{(q+1)/4-1} 
        \Bigl(w + \bigl[w \equiv I{-}1 \pmod{I}\bigr] 
        - \bigl[w \equiv 0 \pmod{I}\bigr]\Bigr) \\
        &= 3\binom{(q+1)/4}{2} + \binom{(q+1)/4}{2} 
        + \frac{(q+1)/4}{I} - \Bigl(\frac{(q+1)/4}{I} - 1\Bigr) = g
        \quad \bigl(I < \tfrac{q+1}{4}\bigr),
    \end{align*}
    and similarly
    \[
        \sum_{v=1}^{a-1} \Bigl\lfloor \frac{f(v)}{a} \Bigr\rfloor
        = \sum_{v=1}^{3((q+1)/4-1)} \!\!\Bigl\lceil \tfrac{v}{3} \Bigr\rceil
        + \sum_{w=1}^{(q+1)/4-1} 
        \Bigl(w + \bigl[w = \tfrac{q+1}{4}-1\bigr]\Bigr) = g
        \quad \bigl(I = \tfrac{q+1}{4}\bigr).
    \]

    \medskip
    \noindent\textbf{Case $I = 1$} (so $a = q-3$). 
    Let $v \in \{1, \ldots, a-1\}$.

    \begin{itemize}
        \item \textbf{Subcase $v = 1$}. Define
        \[
            f(1) \coloneqq \Bigl(\tfrac{q+1}{4} - 3\Bigr)(q+1) + 3q 
            = \Bigl(\tfrac{q+1}{4} + 1\Bigr)a + 1,
        \]
        which lies in $Q(I)$ when $(q+1)/4 \ge 3$. For $(q+1)/4 = 2$, 
        one verifies $f(1) = 13 = (7-1) + 7 \in Q(I)$.

        \item \textbf{Subcase $v \ge 2$}. Write $v - 1 = 4Q + R$ with 
        $Q \in \mathbb{N}$ and $R \in \{0,1,2,3\}$, and define
        \[
            f(v) \coloneqq
            \begin{cases}
                Q(q+1) + (q - 2 + R) & (R \neq 0), \\
                (Q-1)(q+1) + (q-1) + q & (R = 0),
            \end{cases}
        \]
        where $q - 2 + R \in \{q-1, q, q+1\}$ when $R \neq 0$. 
        In both cases $f(v) \in Q(I)$ and $f(v) = \lceil v/4 \rceil \cdot a + v$.
    \end{itemize}

    Therefore,
    \begin{align*}
        \sum_{v=1}^{a-1} \Bigl\lfloor \frac{f(v)}{a} \Bigr\rfloor
        &= \Bigl(\tfrac{q+1}{4} + 1\Bigr) 
        + \sum_{v=2}^{a-1} \Bigl\lceil \tfrac{v}{4} \Bigr\rceil 
        = \Bigl(\tfrac{q+1}{4} + 1\Bigr) 
        + 4\binom{(q+1)/4}{2} - 1 - \tfrac{a}{4} = g. \qedhere
    \end{align*}
\end{proof}

\subsection{The general case}

The preceding subsection illustrates the difficulty of obtaining a general 
result for Weierstrass semigroups at the remaining rational points, as they 
depend heavily on both $m$ and the characteristic of the field.

The general strategy is to consider linear combinations of monomials 
$X^i Y^j$ lying in $\mathcal{L}(k \cdot (X)_\infty)$ with a prescribed 
valuation at $P_{(a,b)}$. For $k = 1$, for instance, we consider the 
functions $Y, \ldots, Y^{\lfloor m/2 \rfloor}, X$. Their $T$-coefficients 
up to $T^{m-1}$ are given by the matrix
\begin{equation}\label{matrix}
\begin{array}{cccccc}
      &   & T^1 & T^2 & \cdots & T^{m-1} \\[4pt]
    Y & : & 1 + \alpha & \alpha & \cdots & 0 \\
    Y^2 & : & 0 & (1 + \alpha)^2 & \cdots & 0 \\
    & & & \vdots & & \\
    Y^{\lfloor m/2 \rfloor} & : & 0 & 0 & \cdots & \star \\
    X & : & m & \binom{m}{2} & \cdots & m
\end{array}
\end{equation}
where
\[
\star = \begin{cases}
    \dfrac{m}{2}(1 + \alpha)\,\alpha^{m/2 - 1} & \text{if } m \text{ is even,} \\[6pt]
    \alpha^{(m-1)/2} & \text{if } m \text{ is odd.}
\end{cases}
\]

If the matrix $M$ formed by these entries has full rank, then there exists 
a nonzero vector $[a_1, a_2, \ldots, b]$ such that the product 
$[a_1, a_2, \ldots, b]\, M$ has its first $\lfloor m/2 \rfloor$ entries 
equal to zero (as produced by Gaussian elimination). This yields a function 
$f \in \mathbb{F}_{q^2}(\mathcal{X}_{m,q})$ satisfying
\[
    \lfloor m/2 \rfloor < v_{P_{(a,b)}}(f) < m.
\]
Consequently, $H(P_{(a,b)})$ contains an element strictly between 
$(q+1) - m$ and $(q+1) - \lfloor m/2 \rfloor$, so 
$H(P_{(a,b)}) \neq H(P_\infty^1)$. Moreover, if $m$ is odd, then 
$Y^2 \in \mathcal{L}((X)_\infty)$, giving 
$(q+1) - 2 \in H(P_{(a,b)}) \setminus H(P_\infty^1)$. If $m$ is even, 
then $H(P_{(a,b)})$ contains an element strictly between 
$(q+1) - m \ge (q+1)/2$ and $(q+1) - m/2$, again showing 
$H(P_{(a,b)}) \neq H(P_\infty^1)$.

One therefore concludes that the Weierstrass semigroup at any place of 
$\mathcal{O}$ differs from that at $P_{(a,b)}$, provided $M$ has full rank 
for every $\alpha$. This will be the key argument in the following section 
for determining the automorphism group of $\mathcal{X}_{m,q}$. Although we 
were unable to prove that the rank is always maximal, we establish it under 
a mild restriction on the characteristic of the field, which we assume 
henceforth.

\begin{proposition}\label{rank-M}
    If $m \ge 3$, $\alpha \neq -1$, and 
    $p \nmid \binom{m}{2 + (m \bmod 2)}$, then the matrix $M$ has full rank.
\end{proposition}

\begin{proof}
    It suffices to exhibit a minor $N$ of $M$ of order $\lfloor m/2 \rfloor + 1$ 
    with nonzero determinant. If $m$ is odd, let $N$ be the submatrix formed 
    by columns $1, \ldots, \lfloor m/2 \rfloor - 1, m-2, m-1$ of $M$; if $m$ 
    is even, take columns $1, \ldots, m/2 - 2, m-3, m-2, m-1$. We claim that
    \[
        \det~N = \begin{cases}
            \dfrac{m}{4}\dbinom{m}{3}
            (1-\alpha)\,\alpha^{m-4}\,(\alpha+1)^{\binom{m/2-1}{2}+1} 
            & (m \text{ even}), \\[8pt]
            \dbinom{m}{2}\alpha^{(m-3)/2}(\alpha+1)^{\binom{(m-1)/2}{2}} 
            & (m \text{ odd}).
        \end{cases}
    \]
    Partition $N$ into blocks
    \[
        N = \begin{bmatrix} A & B \\ C & D \end{bmatrix},
    \]
    where $A$ and $D$ are square, $D$ has order $3 - (m \bmod 2)$, and 
    $B = 0$ (which is possible since $m \ge 3$). Hence 
    $\det~N = \det~A \cdot \det~D$. The matrix $A$ is upper triangular with
    \begin{align*}
        \det~A 
        &= \prod_{i=1}^{\lfloor m/2 \rfloor - 2 + (m \bmod 2)} (\alpha+1)^i 
        = \begin{cases}
            (\alpha+1)^{\binom{m/2-1}{2}} & (m \text{ even}), \\[4pt]
            (\alpha+1)^{\binom{(m-1)/2}{2}} & (m \text{ odd}).
        \end{cases}
    \end{align*}

    \noindent\textbf{Case $m$ odd.} We compute
    \begin{align*}
        \det~D 
        &= \det\begin{bmatrix}
            \frac{m-1}{2}(\alpha+1)\alpha^{(m-3)/2} & \alpha^{(m-1)/2} \\
            \binom{m}{2} & m
        \end{bmatrix} \\
        &= \alpha^{(m-3)/2}\!\left(\binom{m}{2}(\alpha+1) - \binom{m}{2}\alpha\right) 
        = \binom{m}{2}\alpha^{(m-3)/2},
    \end{align*}
    and therefore
    \[
        \det~N = \binom{m}{2}\,\alpha^{(m-3)/2}\,(\alpha+1)^{\binom{(m-1)/2}{2}}.
    \]

    \noindent\textbf{Case $m$ even.} We compute
    \begin{align*}
        \det~D 
        &= \det\begin{bmatrix}
            \binom{m/2-1}{1}(\alpha+1)\alpha^{\frac{m}{2}-2} 
                & \alpha^{\frac{m}{2}-1} & 0 \\[4pt]
            \binom{m/2}{3}(\alpha+1)^3\alpha^{\frac{m}{2}-3} 
                & \binom{m/2}{2}(\alpha+1)^2\alpha^{\frac{m}{2}-2} 
                & \frac{m}{2}(\alpha+1)\alpha^{\frac{m}{2}-1} \\[4pt]
            \binom{m}{3} & \binom{m}{2} & m
        \end{bmatrix} \\
        &= \frac{m^2(m-1)(m-2)}{24}(1-\alpha)\,\alpha^{m-4}(\alpha+1),
    \end{align*}
    and therefore
    \[
        \det~N = \frac{m}{4}\binom{m}{3}
        (1-\alpha)\,\alpha^{m-4}\,(\alpha+1)^{\binom{m/2-1}{2}+1}.
    \]

    Since $\alpha \notin \{0, 1\}$ and $\alpha \neq -1$, we have 
    $\det N \neq 0$, so $M$ has an invertible minor of order 
    $\lfloor m/2 \rfloor + 1$, and the result follows.
\end{proof}

In the case where $p \mid \binom{m}{2 + (m \bmod 2)}$, one could 
examine the remaining minors of $M$ to determine whether any has nonzero 
determinant. Since the characteristic condition in Proposition~\ref{rank-M} 
is mild, we do not pursue this direction further.

From Proposition~\ref{rank-M}, the Weierstrass semigroup at any place of 
$\mathcal{O}$ differs from that at $P_{(a,b)}$ whenever 
$2a^{(q+1)/m} + 1 \neq 0$ and $p \nmid \binom{m}{2+(m \bmod 2)}$. 
If $m > 4$, the same conclusion holds for places with 
$2a^{(q+1)/m} + 1 = 0$, with no restriction on the characteristic.

\begin{proposition}\label{o-linha}
    Let $\mathcal{O}' = \{P_{(a,b)} : 2a^{(q+1)/m} + 1 = 0\}$. 
    If $m > 4$, then $\mathcal{O}'$ is an orbit under the action of the 
    full automorphism group of $\mathcal{X}_{m,q}$.
\end{proposition}

\begin{proof}
    One verifies directly that $\mathcal{O}'$ is an orbit under the group 
    $G$ of Proposition~\ref{automorfismos} or~\ref{prop:bigger-aut}. It 
    therefore suffices to show that the Weierstrass semigroup at any place 
    of $\mathcal{O}'$ differs from that at every other rational place. Let 
    $P_{(a,b)} \in \mathcal{O}'$. With the notation of the beginning of this 
    section, we have $Y = -T^2 + O(T^q)$, so
    \[
        v_{P_{(a,b)}}(Y^i) = 2i \le m < q+1 
        \quad \text{for every } i \in \bigl\{0, \ldots, \lfloor \tfrac{m}{2} \rfloor\bigr\},
    \]
    and hence
    \[
        \Bigl\{q+1,\ (q+1)-2,\ \ldots,\ (q+1) - 2\bigl\lfloor\tfrac{m}{2}\bigr\rfloor\Bigr\} 
        \subseteq H(P_{(a,b)}).
    \]

    \noindent\textbf{Step 1: $H(P_{(a,b)}) \neq H(P_\infty^1)$.}
    \begin{itemize}
        \item \textit{$m$ odd.} Then $(q+1) - 2 \in H(P_{(a,b)})$, 
        whereas $(q+1) - 2 \notin H(P_\infty^1)$, since 
        $2((q+1)-m) \ge q+1$.

        \item \textit{$m < (q+1)/2$ even.} Then 
        $(q+1) - m \in H(P_{(a,b)})$, but $(q+1) - m \notin H(P_\infty^1)$.

        \item \textit{$m = (q+1)/2$ even.} We have
        \[
            (q+1) - 2\!\left(\frac{m}{2} - 1\right) 
            = \frac{q+1}{2} + 2.
        \]
        Since $m \neq 4$, this value lies strictly between 
        $\frac{q+1}{2}$ and $q+1-\frac{m}{2}$, so it belongs to 
        $H(P_{(a,b)})$ but not to $H(P_\infty^1)$.
    \end{itemize}

    \noindent\textbf{Step 2: $H(P_{(a,b)}) \neq H(P_0^1)$.}
    \begin{itemize}
        \item \textit{$m$ odd.} Then 
        $(q+1) - (m-1) \in H(P_{(a,b)})$. Since $m > 4$ implies 
        $m - 1 > \lfloor m/2 \rfloor$, we have 
        $(q+1) - (m-1) \notin H(P_0^1)$.

        \item \textit{$m$ even.} Then 
        $(q+1) - (m-2) \in H(P_{(a,b)})$. Since $m > 4$ implies 
        $m - 2 > \lfloor m/2 \rfloor$, we have 
        $(q+1) - (m-2) \notin H(P_0^1)$.
    \end{itemize}

    \noindent\textbf{Step 3: $H(P_{(a,b)}) \neq H(P_{(c,d)})$ for 
    $P_{(c,d)} \notin \mathcal{O} \cup \mathcal{O}'$.}
    The Weierstrass semigroup at the places of $\mathcal{O}$ shows that 
    the linear series $|(q+1)P|$ has dimension $\lfloor m/2 \rfloor + 2$ 
    for every rational place $P$. Applying~\eqref{div-Gk} to 
    $Y, \ldots, Y^{\lfloor m/2 \rfloor}$ and using the maximality of the 
    curve gives
    \[
        \Bigl\{q+1-\bigl\lfloor\tfrac{m}{2}\bigr\rfloor,\ \ldots,\ q+1\Bigr\} 
        \subset H(P_{(c,d)}),
    \]
    so $H(P_{(c,d)})$ has exactly one nonzero element smaller than 
    $q+1 - \lfloor m/2 \rfloor$.
    \begin{itemize}
        \item \textit{$m$ even.} Applying~\eqref{div-Gk} to $Y^{m/2}$ 
        and $Y^{m/2-1}$ gives 
        $q+1-m,\, q+1-(m-2) \in H(P_{(a,b)})$; since $m > 4$, both are 
        smaller than $q+1-\lfloor m/2 \rfloor$.

        \item \textit{$m$ odd, $m \neq 5$.} We obtain 
        $q+1-(m-1),\, q+1-(m-3) \in H(P_{(a,b)})$, both smaller than 
        $q+1-\lfloor m/2 \rfloor$.

        \item \textit{$m = 5$.} Proposition~\ref{rank-M} yields a 
        rational function with $T$-valuation $3$ at $P_{(a,b)}$, giving
        \[
            \{q-3,\, q-1,\, q,\, q+1\} \subset H(P_{(a,b)}) 
            \quad \text{and} \quad 
            \{q-2,\, q-1,\, q,\, q+1\} \subset H(P_{(c,d)}).
        \]
    \end{itemize}
    In all cases $H(P_{(a,b)}) \neq H(P_{(c,d)})$, completing the proof.
\end{proof}

\section{Automorphism group}\label{sec:automorphisms}

Since $\mathcal{X}_{m,q}$ is a maximal curve over $\mathbb{F}_{q^2}$, 
by \cite[Theorem~3.10]{GSY}, its automorphism group $\Aut(\mathcal{X}_{m,q})$ 
is defined over $\mathbb{F}_{q^2}$. In particular, the orbit of any 
$\mathbb{F}_{q^2}$-rational place under $\Aut(\mathcal{X}_{m,q})$ is 
contained in the set of $\mathbb{F}_{q^2}$-rational places.
Moreover, places in the same orbit share the same Weierstrass semigroup 
\cite[Lemma~3.5.2]{stichtenoth}. Throughout this section we assume 
$p \nmid \binom{m}{2+(m \bmod 2)}$, which provides sufficient information 
on the orbits of rational places to identify the automorphism group with 
the group given in Proposition~\ref{automorfismos} or~\ref{prop:bigger-aut}, 
according to whether or not $m = (q+1)/2$.

\subsection{\texorpdfstring{Case $m = 4$}{Case m = 4}}

This case requires separate treatment: by Remark~\ref{rmk:m4}, we have 
$H(P_{(a,b)}) = H(P)$ for any rational place with $2a^{(q+1)/4} + 1 = 0$ 
and $P \in \mathcal{O}_0 \cup \mathcal{O}_m$. As a finite number of automorphism groups may be checked by an algorithm, we therefore assume $q > 31$ for the 
remainder of this subsection.

\begin{lemma}\label{sylow4}
    $\Aut(\mathcal{X}_{4,q})$ has no nontrivial $p$-Sylow subgroup.
\end{lemma}

\begin{proof}
    The results on Weierstrass semigroups from 
    Sections~\ref{sec:Weierstrass} and~\ref{sec:weierstrass-other} imply 
    that $\mathcal{O}_\infty$ is an orbit under $\Aut(\mathcal{X}_{4,q})$. 
    By the Orbit-Stabilizer theorem, any nontrivial $p$-Sylow subgroup $S$ partitions 
    $\mathcal{X}_{4,q}(\mathbb{F}_{q^2})$ into orbits of $p$-power size. 
    Since $p \nmid |\mathcal{O}_\infty| = \frac{q+1}{2}$, the subgroup $S$ 
    must fix some point of $\mathcal{O}_\infty$. However, maximal curves 
    have $p$-rank zero \cite[Theorem~10.1]{HKT}, and every automorphism of 
    order $p$ of a curve of $p$-rank zero has exactly one fixed point 
    \cite[Lemma~11.129]{HKT}. Since $|\mathcal{O}_\infty| \not\equiv 1 
    \pmod{p}$, the subgroup $S$ would have at least two fixed points in 
    $\mathcal{O}_\infty$, a contradiction.
\end{proof}

Although it is a priori possible for a place $P_{(a,b)}$ with 
$2a^{(q+1)/4} + 1 = 0$ to lie in the orbit of $P_{(a,0)}$, the assumption 
$q > 31$ rules this out.

\begin{lemma}\label{lem:orbit-Om}
    The orbit of $P_{(a,0)}$ under $\Aut(\mathcal{X}_{4,q})$ is 
    $\mathcal{O}_0 \cup \mathcal{O}_m$.
\end{lemma}

\begin{proof}
    Suppose for contradiction that some place $P_{(a,b)}$ with 
    $2a^{(q+1)/4} + 1 = 0$ lies in the orbit of $P_{(a,0)}$ (the 
    Weierstrass semigroup data show this is the only other possibility). 
    The orbit of $P_{(a,b)}$ under $G$ (Proposition~\ref{automorfismos}) 
    is $\mathcal{O}' = \{P_{(\tilde{a},\tilde{b})} : 2\tilde{a}^{(q+1)/4} 
    + 1 = 0\}$, so the orbit of $P_{(a,0)}$ under the full automorphism 
    group would consist of $\frac{(q+1)(q+3)}{4}$ places.

    By the Orbit-Stabilizer theorem applied to $G$, we have 
    $|G_{P_{(a,0)}}| = q+1$. Since 
    $|\Aut(\mathcal{X}_{4,q})_{P_{(a,0)}}| \ge |G_{P_{(a,0)}}|$, 
    applying Orbit-Stabilizer to the full automorphism group gives 
    $|\Aut(\mathcal{X}_{4,q})| \ge \frac{(q+1)^2(q+3)}{4}$. For $q > 31$ 
    this exceeds $84(g-1)$, so \cite[Theorem~11.56]{HKT} forces 
    $p \mid |\Aut(\mathcal{X}_{4,q})|$, contradicting Lemma~\ref{sylow4}.
\end{proof}

\begin{proposition}\label{aut1}
    Let $G$ be as in Proposition~\ref{automorfismos}. Then 
    $\Aut(\mathcal{X}_{4,q}) = G$.
\end{proposition}

\begin{proof}
    Suppose $|\Aut(\mathcal{X}_{4,q})| > |G|$. By Lemma~\ref{lem:orbit-Om}, $\mathcal{O}_0 \cup \mathcal{O}_m$ is an 
    orbit under $\Aut(\mathcal{X}_{4,q})$. Applying the Orbit-Stabilizer 
    theorem to both $G$ and $\Aut(\mathcal{X}_{4,q})$ shows that there 
    exists $\gamma \in \Aut(\mathcal{X}_{4,q})_{P_{(a,0)}} \setminus 
    G_{P_{(a,0)}}$. Let $C = \langle \sigma \rangle$, where 
    $\sigma \colon (x,y) \mapsto (x, \delta y)$ and $\delta$ is a primitive 
    $(q+1)$-th root of unity. Since $p \nmid |\Aut(\mathcal{X}_{4,q})|$, 
    the stabilizer $\Aut(\mathcal{X}_{4,q})_{P_{(a,0)}}$ is cyclic 
    \cite[Theorem~11.49]{HKT}, so $\gamma$ commutes with $C$. Hence for 
    every point $P$ fixed by $C$,
    \begin{equation}\label{sigma-gamma}
			\sigma^i(\gamma(P)) = \gamma(\sigma^i(P)) = \gamma(P),
		\end{equation}
    so $\gamma$ permutes the fixed points of $C$, namely $\mathcal{O}_m$. 
    Since $\mathcal{O}_0 \cup \mathcal{O}_m$ and $\mathcal{O}_\infty$ are 
    both orbits under $\Aut(\mathcal{X}_{4,q})$, the automorphism $\gamma$ 
    preserves $\mathcal{O}_0$ and $\mathcal{O}_\infty$, and therefore 
    preserves the divisors of $x$ and $y$. It follows that 
    $\gamma(x,y) = (\mu_1 x, \mu_2 y)$ for some 
    $\mu_1, \mu_2 \in \mathbb{F}_{q^2}^*$, giving $\gamma \in G$, 
    a contradiction.
\end{proof}

\begin{remark}
Regarding the remaining values of $q$, 
the proprietary software Magma affirms
that $|\Aut(\mathcal{X}_{4,7})| = 192$, and that 
$\Aut(\mathcal{X}_{4,q}) = G$ (as in Proposition~\ref{automorfismos}) for 
$q \in \{11, 19, 23, 27, 31\}$.
\end{remark}

	\subsection{\texorpdfstring
    {Case $m$ odd, or $m < (q+1)/2$ with $m > 4$ and $p > 2$}
    {Case m odd, or m < (q+1)/2, m > 4, p > 2}}

Under these hypotheses, together with our assumption on the characteristic, 
the sets $\mathcal{O}_\infty$ and $\mathcal{O}_0 \cup \mathcal{O}_m$ are 
orbits under $\Aut(\mathcal{X}_{m,q})$.

\begin{lemma}
    $\Aut(\mathcal{X}_{m,q})$ has no nontrivial $p$-Sylow subgroup.
\end{lemma}

\begin{proof}
    By the same argument as in Lemma~\ref{sylow4}, if $\Aut(\mathcal{X}_{m,q})$ 
    had a nontrivial $p$-Sylow subgroup $S$, its unique fixed point would lie in 
    $\mathcal{O}_\infty$. Since $\mathcal{O}_0 \cup \mathcal{O}_m$ is an 
    orbit under the automorphism group, this would force 
    $p \mid |\mathcal{O}_0 \cup \mathcal{O}_m| = 2\,\frac{q+1}{m}$, 
    a contradiction.
\end{proof}

The argument of Proposition~\ref{aut1} now carries over verbatim to give 
the following.

\begin{proposition}
    Let $G$ be as in Proposition~\ref{automorfismos}. Then 
    $\Aut(\mathcal{X}_{m,q}) = G$.
\end{proposition}

\subsection{\texorpdfstring{Case $p = 2$}{Case p = 2}}

We show that $\Aut(\mathcal{X}_{m,q}) = G$ also when the characteristic 
of $\mathbb{F}_q$ is even. The argument is similar to the previous cases, 
but requires adjustment since $\Aut(\mathcal{X}_{m,q})$ now admits a 
$2$-Sylow subgroup.

\begin{lemma}
    Every nontrivial $2$-Sylow subgroup $S$ of $\Aut(\mathcal{X}_{m,q})$ has order $2$.
\end{lemma}

\begin{proof}
    When $p = 2$ we necessarily have $m$ odd, so $\mathcal{O}_\infty$ and 
    $\mathcal{O}_0 \cup \mathcal{O}_m$ are orbits under 
    $\Aut(\mathcal{X}_{m,q})$. Since $|\mathcal{O}_\infty| = \frac{q+1}{m}$ 
    is odd, the unique fixed point of $S$ lies in $\mathcal{O}_\infty$, so 
    $S$ acts freely on $\mathcal{O}_0 \cup \mathcal{O}_m$, giving 
    $|S| \mid |\mathcal{O}_0 \cup \mathcal{O}_m| = 2\,\frac{q+1}{m}$. 
    Since $\frac{q+1}{m}$ is odd, we have $4 \nmid |\mathcal{O}_0 \cup 
    \mathcal{O}_m|$, and hence $|S| = 2$.
\end{proof}

\begin{proposition}
    Let $G$ be as in Proposition~\ref{automorfismos}. Then 
    $\Aut(\mathcal{X}_{m,q}) = G$.
\end{proposition}

\begin{proof}
    Suppose $|\Aut(\mathcal{X}_{m,q})| > |G|$ and let 
    $\gamma \in \Aut(\mathcal{X}_{m,q})_{P_{(a,0)}} \setminus 
    G_{P_{(a,0)}}$. Since every nontrivial $2$-Sylow subgroup has order $2$ and 
    $2 \mid |G|$, we may assume $\gamma$ has odd order. Let 
    $C = \langle \sigma \rangle$ be as in Proposition~\ref{aut1}. By 
    \cite[Theorem~11.49]{HKT} the tame part of $\Aut(\mathcal{X}_{m,q})$ 
    is cyclic, so $\gamma$ commutes with $C$, and the remainder of the 
    proof is identical to that of Proposition~\ref{aut1}.
\end{proof}

\subsection{\texorpdfstring{Case $m = (q+1)/2$ even}
    {Case m = (q+1)/2 even}}

As noted at the end of Section~\ref{sec:Weierstrass}, $\mathcal{O}$ is 
contained in an orbit under $\Aut(\mathcal{X}_{(q+1)/2,q})$. The single case 
$(m, q) = (4, 7)$ may be checked by an algorithm, so we assume $q \ge 11$. By the 
general-case discussion of the preceding section, $\mathcal{O}$ is itself 
an orbit under this group.

\begin{proposition}
    Let $G$ be as in Proposition~\ref{prop:bigger-aut}. If 
    $q \equiv 3 \pmod{4}$ and $q \ge 11$, then 
    $\Aut(\mathcal{X}_{(q+1)/2,q}) = G$.
\end{proposition}

\begin{proof}
    We first show $p \nmid |\Aut(\mathcal{X}_{(q+1)/2,q})|$ by an argument analogous 
    to the case $m < \frac{q+1}{2}$. By Proposition~\ref{o-linha}, the set 
    $\mathcal{O}' = \{P_{(a,b)} : 2a^2 + 1 = 0\}$ is an orbit under 
    $\Aut(\mathcal{X}_{(q+1)/2,q})$ of cardinality $2(q+1) \equiv 2 \pmod{p}$. Thus a nontrivial
    $p$-Sylow subgroup of $\Aut(\mathcal{X}_{(q+1)/2,q})$ would have at least two fixed 
    points in $\mathcal{O}'$, contradicting 
    \cite[Lemma~11.129]{HKT}.

    The rest follows as in Proposition~\ref{aut1}. Assuming 
    $|\Aut(\mathcal{X}_{(q+1)/2,q})| > |G|$, we obtain $\gamma$ acting on 
    $\mathcal{O}_m$. Since $\mathcal{O}_0$ and $\mathcal{O}_\infty$ are 
    orbits under $\langle \sigma \rangle$ (where 
    $\sigma \colon (x,y) \mapsto (x, \delta y)$ with $\delta$ a primitive 
    $(q+1)$-th root of unity) and $\mathcal{O}$ is an orbit under 
    $\Aut(\mathcal{X}_{(q+1)/2,q})$, an analysis analogous to 
    \eqref{sigma-gamma} shows that $\gamma$ either fixes or interchanges 
    $\mathcal{O}_0$ and $\mathcal{O}_\infty$. Since these sets have 
    different cardinalities, $\gamma$ fixes each of 
    $\mathcal{O}_0$, $\mathcal{O}_\infty$, and $\mathcal{O}_m$, and the 
    argument of Proposition~\ref{aut1} gives $\gamma \in G$, a contradiction.
\end{proof}
	\section*{Acknowledgements}
João Paulo Guardieiro was partially supported by FAPESP (Brazil), 
grant no.\ 2024/19443-4. Yuri da Silva was partially supported by 
CAPES (Brazil) -- Finance Code 001. Saeed Tafazolian was partially 
supported by CNPq grant no.\ 302774/2025-4, FAEPEX grant no.\ 3485/25, 
and FAPESP grant no.\ 2024/00923-6. The authors thank Daniel Cariello 
for his assistance with the rank computation for the matrix in 
\eqref{matrix}.

\end{document}